\documentclass[11pt]{article}

\usepackage{preamble}  

\graphicspath{{./figures/}}  
\title{Subspace-constrained randomized coordinate descent for linear systems with good low-rank matrix approximations}

\author{
    Jackie Lok\thanks{    
        ORFE Department,
        Princeton University
        (\texttt{jackie.lok@princeton.edu}, \texttt{elre@princeton.edu}).
    }
    \and
    Elizaveta Rebrova\footnotemark[1]
}

\date{}

\begin{document}

\maketitle

\begin{abstract}
The randomized coordinate descent (RCD) method is a classical algorithm with simple, lightweight iterations that is widely used for various optimization problems, including the solution of positive semidefinite linear systems. As a linear solver, RCD is particularly effective when the matrix is well-conditioned; however, its convergence rate deteriorates rapidly in the presence of large spectral outliers.
In this paper, we introduce the subspace-constrained randomized coordinate descent (SC-RCD) method, in which the dynamics of RCD are restricted to an affine subspace corresponding to a column Nystr\"{o}m approximation, efficiently computed using the recently analyzed RPCholesky algorithm. We prove that SC-RCD converges at a rate that is unaffected by large spectral outliers, making it an effective and memory-efficient solver for large-scale, dense linear systems with rapidly decaying spectra, such as those encountered in kernel ridge regression. Experimental validation and comparisons with related solvers based on coordinate descent and the conjugate gradient method demonstrate the efficiency of SC-RCD.
Our theoretical results are derived by developing a more general subspace-constrained framework for the sketch-and-project method. This framework, which may be of independent interest, generalizes popular algorithms such as randomized Kaczmarz and coordinate descent, and provides a flexible, implicit preconditioning strategy for a variety of iterative solvers.
\end{abstract}

\section{Introduction} \label{sec:intro}

The problem of solving large-scale systems of linear equations $\mA \vx = \vb$ is ubiquitous in machine learning and scientific computing.
The growing size of datasets presents new challenges and demands on the algorithms used. For instance, the entire matrix $\mA$ may not fit into memory, which means that memory-efficient iterative methods are practically advantageous or even necessary. Furthermore, if evaluating entries of the matrix $\mA$ is associated with a nontrivial cost, then it may also be desirable to seek algorithms that are based on accessing small parts of the matrix or are more easily parallelizable, compared to algorithms based on other primitives such as matrix--vector products.

\emph{A key motivation for this work is the development of an efficient solver for dense, positive semidefinite \emph{(psd)} linear systems that requires a limited number of entry accesses at a time.}
As a primary example, consider the problem of kernel ridge regression (KRR)~\cite{ScholkopfSmola2002}. Given $n$ data points, this is a method for nonparametric regression that is equivalent to solving a linear system $(\mK + \lambda \mI) \vx = \vy$, where $\mK \in \reals^{n \times n}$ is a positive definite kernel matrix, which has entries measuring the similarity between each pair of input feature vectors, and $\lambda \geq 0$ is a chosen regularization parameter. The matrix $\mK$ is typically dense, and the regularized system $\mK + \lambda \mI$ can be very poorly conditioned for small values of $\lambda$.
Loading the entire kernel matrix in memory requires storing $O(n^2)$ entries, and solving the linear system using a standard direct method requires $O(n^3)$ arithmetic operations; for large $n$, both can be infeasible. 

\paragraph{Good low-rank matrix approximations.}
Fortunately, matrices arising from many real-life datasets often possess fast spectral decay~\cite{UdellTownsend2019}, and so can be well-approximated by much smaller low-rank matrices.
A particularly effective method for computing a good low-rank approximation of a psd $n \times n$ matrix $\mA$ is the \emph{randomly pivoted Cholesky} (RPCholesky) algorithm, recently proposed and analyzed by Chen et al.~\cite{ChenEtAl2024}.
Given an approximation rank parameter $d \geq 0$, RPCholesky efficiently finds a near-optimal low-rank approximation of $\mA$ using only $O(d^2 n)$ arithmetic operations, $O(dn)$ entry evaluations, and $O(dn)$ storage. It can also be practically accelerated using block computations and rejection sampling~\cite{EpperlyEtAl2024rejection}.

Briefly, the RPCholesky algorithm selects a random set of $d$ pivots $\mathcal{S} \subseteq [n] := \{ 1, 2, \ldots, n \}$ using adaptive diagonal sampling, and outputs a (column) \emph{Nystr\"{o}m approximation}
\[
    \mA \langle \mathcal{S} \rangle = \mF \mF^{\tran}
\]
of $\mA$ in factorized form, where $\mF \in \reals^{n \times d}$ is a lower-triangular matrix after permuting the rows to bring the pivots $\mathcal{S}$ to the top, and the residual matrix $\mA - \mA \langle \mathcal{S} \rangle$ has trace-norm error comparable with the best rank-$r$ approximation of $\mA$ (from truncated SVD) for some $r \approx d$.
More precisely, the approximation quality is described by 

\begin{theorem}[{\cite[Theorem~5.1]{ChenEtAl2024}}] \label{thm:rpcholesky_error}
Let $\mA \in \reals^{n \times n}$ be a psd matrix with eigenvalues $\lambda_1(\mA) \geq \ldots \geq \lambda_n(\mA) \geq 0$, and denote $\log_+(t) = \max \{ \log t, 0 \}$ for $t > 0$.
Suppose that for some fixed integer $r \geq 0$ and real $\delta > 0$, the approximation rank parameter $d$ satisfies
\[
    d \geq \frac{r}{\delta} + \min\left\{ r \log\left( \frac{1}{\delta \eta_r} \right),\, r + r \log_+\left( \frac{2^r}{\delta} \right) \right\}, \quad \text{where} \quad
    \eta_r := \frac{\sum_{i > r} \lambda_i(\mA)}{\sum_{i=1}^n \lambda_i(\mA)}.
\]
Then, the rank-$d$ Nystr\"{o}m approximation $\mA \langle \mathcal{S} \rangle$ output by RPCholesky satisfies
\[
    \E[\tr(\mA - \mA \langle \mathcal{S} \rangle)] \leq (1 + \delta) \cdot \sum_{i > r} \lambda_i(\mA).
\]
\end{theorem}

Among other applications, matrix approximation techniques based on RPCholesky have recently been used to build efficient preconditioners for conjugate gradient-based linear solvers for KRR~\cite{DiazEtAl2023robust}.

\paragraph{Randomized block coordinate descent.}
A convenient and widely-used optimization approach that requires a limited number of local entry accesses is the coordinate descent method~\cite{LeventhalLewis2010, Nesterov2012}.
For solving the psd system $\mA \vx = \vb$, the algorithm has basic iterations of the form
\begin{equation} \label{eq:rcd_update}
    \vx^{k+1} = \vx^k - \frac{(\mA_{:,j})^{\tran} \vx^k - \vb_j}{\mA_{j,j}} \ve_j,
\end{equation}
where $\ve_j \in \reals^n$ denotes the $j$th standard basis vector in $\reals^n$.
Observe that the updates~\eqref{eq:rcd_update} are very lightweight, requiring only access to a single column of $\mA$ and $O(n)$ arithmetic operations.
Leventhal and Lewis~\cite{LeventhalLewis2010} showed that if the coordinate $j \in [n]$ is sampled with probability $\mA_{j,j} / \tr(\mA)$, then the corresponding \emph{randomized coordinate descent} (RCD) method converges linearly with a rate that depends on the spectrum of $\mA$. More precisely, if $\vx^*$ is any solution of $\mA \vx = \vb$ and $\lambda_{\mathrm{min}}^+(\mA)$ is the smallest non-zero eigenvalue of $\mA$, then
\begin{equation} \label{eq:rcd_conv_bdd}
    \E \norm{\vx^k - \vx^*}_{\mA}^2 \leq \left( 1 - \frac{\lambda^+_{\mathrm{min}}(\mA)}{\tr(\mA)} \right)^k \cdot \norm{\vx^0 - \vx^*}_{\mA}^2,
\end{equation}
where $\norm{\vz}_{\mA} = \sqrt{\vz^{\tran} \mA \vz}$ is the norm induced by $\mA$. The iteration~\eqref{eq:rcd_update} readily generalizes to the \emph{block coordinate descent} update
\begin{equation} \label{eq:block_rcd_update}
    \vx^{k+1} = \vx^k - \ve_{\mathcal{J}} (\mA_{\mathcal{J}, \mathcal{J}})^{\dagger} \bigl( (\mA_{:, \mathcal{J}})^{\tran} \vx^k - \vb_{\mathcal{J}} \bigr),
\end{equation}
where $\ve_{\mathcal{J}} \in \reals^{n \times |\mathcal{J}|}$ denotes the matrix whose columns are the standard basis vectors in $\reals^n$ corresponding to a subset $\mathcal{J} \subseteq [n]$ (i.e., only the coordinates in $\mathcal{J}$ are updated).
However, with block size $\ell = |\mathcal{J}| > 1$, bounds on the convergence rate have only been obtained in special cases: e.g., if given a well-conditioned partitioning of the columns~\cite{NeedellTropp2014, WuNeedell2018}, or if the blocks $\mathcal{J}$ are sampled with a special distribution (which is not easy to compute)~\cite{rodomanov2020block, mutny2020block}.

The convergence rate bound~\eqref{eq:rcd_conv_bdd} implies that if the matrix $\mA$ is very well-conditioned, in the sense that $\lambda^+_{\mathrm{min}}(\mA)$ is of the same order as the maximum eigenvalue and so $\lambda^+_{\mathrm{min}}(\mA) / \tr(\mA) = O(n)$, then $O(n \log(1 / \epsilon))$ iterations and $O(n^2 \log(1 / \epsilon))$ arithmetic operations suffice to compute an approximate solution with relative error in $\norm{\cdot}_{\mA}$ bounded by $\epsilon$.
However, if $\mA$ is not so well-conditioned, then the bound rapidly deteriorates and RCD can easily be demonstrated to converge very slowly. 
Recent works such as \cite{DerezinskiEtAl2024fine, DerezinskiEtAl2025beyond} show that properly designed block generalizations of RCD can produce competitive linear solvers with implicit preconditioning properties for systems with large spectral outliers.

\paragraph{Key question.}
Motivated by these observations, we focus on the following question: \emph{Can we use an efficiently-computable low-rank approximation to improve the rate of convergence for solving a psd system $\mA \vx = \vb$ with a lightweight, memory-efficient iterative algorithm such as randomized block coordinate descent, especially when $\mA$ exhibits rapid spectral decay?}

\subsection{Notation}

We write vectors and matrices in boldface. We denote the identity matrix by $\mI$, and the Moore--Penrose pseudoinverse of a matrix $\mA$ by $\mA^{\dagger}$.
Given a subset $\mathcal{J} \subseteq [n] := \{ 1, 2, \ldots, n \}$, we denote the column submatrix of $\mA \in \reals^{n \times n}$ indexed by $\mathcal{J}$ by $\mA_{:, \mathcal{J}}$. Similarly, we indicate the row and principal submatrix indexed by $\mathcal{J}$ by $\mA_{\mathcal{J},:}$ and $\mA_{\mathcal{J}, \mathcal{J}}$ respectively.
Furthermore, we write $\ve_{\mathcal{J}} \in \reals^{n \times |\mathcal{J}|}$ to denote the matrix whose columns are the standard basis vectors in $\reals^n$ corresponding to the coordinates in $\mathcal{J}$.
We denote the eigenvalues of a symmetric matrix $\mA \in \reals^{n \times n}$ in non-increasing order by $\lambda_1(\mA) \geq \ldots \geq \lambda_{n}(\mA) = \lambda_{\mathrm{min}}(\mA)$, and the smallest non-zero eigenvalue by $\lambda_{\mathrm{min}}^+(\mA)$.
We use the Loewner order $\mA \preceq \mB$ for symmetric matrices $\mA$, $\mB$ to mean that $\mB - \mA$ is positive semidefinite (psd).
Given a psd matrix $\mA \succeq \mzero$, the $\mA$-(semi)norm is defined by $\norm{\vz}_{\mA} = \sqrt{\vz^{\tran} \mA \vz}$.

\subsection{Main results} \label{sec:main_results}

In this section, we will describe an efficient algorithm based on randomized block coordinate descent that provides a solution to our key question. We will also describe a more general subspace-constrained framework that we developed to obtain our theoretical results.

\paragraph{\textit{I. Subspace-constrained randomized coordinate descent (SC-RCD)}.}
We show that a column Nystr\"{o}m approximation $\mA \langle \mathcal{S} \rangle := \mA_{:,\mathcal{S}} (\mA_{\mathcal{S},\mathcal{S}})^{\dagger} \mA_{\mathcal{S},:}$ of the psd matrix $\mA \in \reals^{n \times n}$ can be combined with randomized block coordinate descent by constraining the iterates $\vx^k$ of RCD within the following affine subspace, parameterized by the pivot set $\mathcal{S} \subseteq [n]$:
\[
    \vx^k \in \{ \vx \in \reals^n: \mA_{\mathcal{S},:} \vx = \vb_{\mathcal{S}} \} \quad \text{for all } k \geq 0.
\]
We show that when the blocks $\mathcal{J} \subseteq [n]$ in each iteration consist of $\ell$ coordinates independently sampled with probability proportional to the diagonal of the \emph{residual matrix}
\begin{equation}\label{eq:a-circ-def}
    \mA^{\circ} := \mA - \mA \langle \mathcal{S} \rangle,
\end{equation}
then the convergence rate depends on the spectrum of $\mA^{\circ}$ instead of $\mA$.
That is, \emph{the restriction of the dynamics to the subspace corresponding to a given low-rank approximation effectively preconditions the system}, and a significant improvement in the convergence rate can be achieved when a good low-rank matrix approximation is efficiently computable, using an algorithm such as RPCholesky to select $\mathcal{S}$.

\paragraph{The SC-RCD method.}
Given an approximation rank parameter $d \geq 1$ and block size parameter $\ell \geq 1$, the resulting algorithm, which we refer to as \emph{subspace-constrained randomized coordinate descent} (SC-RCD)\footnote{If the block size $\ell > 1$, then the updates of the algorithm are more accurately described as randomized \emph{block} coordinate descent, but we will use the same acronym SC-RCD for brevity.}, can be described as follows:

\smallskip

\begin{itemize}[leftmargin=2em]
    \item A rank-$d$ Nystr\"{o}m approximation $\mA \langle \mathcal{S} \rangle = \mF \mF^{\tran}$ with corresponding pivots $\mathcal{S}$ is efficiently computed using an algorithm such as RPCholesky, and an initial iterate $\vx^0$ is obtained by solving $\mA_{\mathcal{S},:} \vx^0 = \vb_{\mathcal{S}}$.
    Let $\mC := (\mA_{\mathcal{S},\mathcal{S}})^{\dagger} \mA_{\mathcal{S},:} \in \reals^{d \times n}$.

    \smallskip
    
    \item In the $(k+1)$st iteration, given the iterate $\vx^k$ and corresponding residual vector $\vr^k = \mA \vx^k - \vb$, we form a random subset $\mathcal{J} = \{ j_1, \ldots, j_\ell \} \subseteq [n] \setminus \mathcal{S}$ of $\ell$ coordinates, each sampled independently with probability proportional to the corresponding diagonal entry of the residual matrix $\mA^{\circ} = \mA - \mA \langle \mathcal{S} \rangle$. Then, we compute
    \[
        \valpha^k = (\mA^{\circ}_{\mathcal{J}, \mathcal{J}})^{\dagger} \vr^k_{\mathcal{J}} \in \reals^\ell 
        \quad\text{and}\quad
        \vbeta^k = \mC_{:, \mathcal{J}} \valpha^k \in \reals^d,
    \]
    and use these vectors to update the coordinates of $\vx^k$ in $\mathcal{J}$ and $\mathcal{S}$:
    \begin{equation} \label{eq:sc-rcd_update_intro}
        \vx^{k+1}_{\mathcal{J}} = \vx^{k}_{\mathcal{J}} - \valpha^k,
        \quad
        \vx^{k+1}_{\mathcal{S}} = \vx^{k}_{\mathcal{S}} + \vbeta^k.
    \end{equation}
    The other coordinates of $\vx^k$ remain unchanged. Furthermore, the residual vector is updated by 
    \begin{equation} \label{eq:sc-rcd_update_resid_intro}
        \vr^{k+1} = \vr^k - \mA^{\circ}_{:, \mathcal{J}} \valpha^k.
    \end{equation}
\end{itemize}

\begin{remark}
If we denote a solution of the system $\mA \vx = \vb$ by $\vx^*$, then it can be shown that given $\vx^k$, the next iterate $\vx^{k+1}$ produced by the update~\eqref{eq:sc-rcd_update_intro} satisfies
\begin{equation} \label{eq:sc-rcd_descent_interpretation}
    \vx^{k+1} = \argmin_{\vx \in \reals^n}
    \norm{\vx - \vx^*}_{\mA}
    \quad \text{such that} \quad \vx = \vx^k + \ve_{\mathcal{J}} \valpha + \ve_{\mathcal{S}} \vbeta,
\end{equation}
where only $\valpha \in \reals^{\ell}$ and $\vbeta \in \reals^d$ are free to vary, i.e., we minimize the $\mA$-norm error with respect to the coordinates in $\mathcal{J}$ and $\mathcal{S}$.
Without the subspace constraint, we would only optimize over the coordinates in the sampled subset $\mathcal{J}$, which corresponds to the usual block coordinate descent update~\eqref{eq:block_rcd_update}.
With the additional subspace constraint, the coordinates in $\mathcal{S}$ must also be modified to ensure that the iterates continue to satisfy $\mA_{\mathcal{S},:} \vx = \vb_{\mathcal{S}}$.
What distinguishes the iteration-dependent subset $\mathcal{J}$ from the fixed subset $\mathcal{S}$ is that we know that $\vx^k$ satisfies $\mA_{\mathcal{S},:} \vx^k = \vb_{\mathcal{S}}$ a priori, which allows for the update for $\mathcal{S}$ to be computed more efficiently using the update for $\mathcal{J}$.
\end{remark}

\begin{algorithm}[!htb]
\caption{Subspace-constrained randomized coordinate descent (SC-RCD)} \label{alg:sc-rcd}
\begin{algorithmic}[1]
    \Require{Psd matrix $\mA \in \reals^{n \times n}$, vector $\vb \in \reals^{n}$, approximation rank $d$, block size $\ell$}
    \Ensure{Approximate solution $\vx \in \reals^n$ of $\mA \vx = \vb$, residual vector $\vr = \mA \vx - \vb \in \reals^n$}
    \State Compute pivot set $\mathcal{S} \subseteq [n]$ and partial pivoted Cholesky factor $\mF \in \reals^{n \times d}$ defining the Nystr\"{o}m approximation $\mA \langle \mathcal{S} \rangle = \mF \mF^{\tran}$  \Comment{E.g., RPCholesky~\cite{ChenEtAl2024, EpperlyEtAl2024rejection}}
    \State Compute $\vbeta \leftarrow (\mF_{\mathcal{S},:})^{-\tran} (\mF_{\mathcal{S},:})^{-1} \vb_{\mathcal{S}} \in \reals^{d}$  \Comment{Triangular solves with $\mF$} \label{alg:sc-rcd_initial_solve0}
    \State Set $\vx \leftarrow \vzero_{n \times 1}$, $\vx_{\mathcal{S}} \leftarrow \vbeta$, and $\vr \leftarrow \mA_{:,\mathcal{S}} \vbeta - \vb \in \reals^n$  \Comment{Initial iterate (\ref{rmk:sc-rcd_implementation_initial})} \label{alg:sc-rcd_initial_solve}
    \State Compute $\mC \leftarrow (\mF_{\mathcal{S},:})^{-\tran} \mF^{\tran} \in \reals^{d \times n}$  \Comment{Triangular solves with $\mF$ (\ref{rmk:sc-rcd_implementation_compute_B})} \label{alg:sc-rcd_compute_B}
    \State Set $\vp \leftarrow \vzero_{n \times 1}$, compute $\vp_j \leftarrow \mA_{j,j} - \mF_{j,:} (\mF_{j,:})^{\tran}$ for $j \in [n] \setminus \mathcal{S}$, and normalize: $\vp \leftarrow \vp / \sum_{j} \vp_j$
    \For{$k = 1, 2, \ldots$}
        \State Sample $\mathcal{J} = \{ j_1, \ldots, j_\ell \}$ with $j_1, \ldots, j_\ell \sim \vp$ i.i.d.  \Comment{Possibly without replacement (\ref{rmk:sc-rcd_implementation_sampling_withoutrep})} \label{alg:sc-rcd_sampling_step}
        \State Solve $\bigl( \mA_{\mathcal{J}, \mathcal{J}} - \mF_{\mathcal{J},:} (\mF_{\mathcal{J},:})^{\tran} \bigr) \valpha = \vr_{\mathcal{J}}$ for $\valpha \in \reals^\ell$  \Comment{Possibly inexactly (\ref{rmk:sc-rcd_implementation_inexact})} \label{alg:sc-rcd_projection_step}
        \State $\vbeta \leftarrow \mC_{:, \mathcal{J}} \valpha \in \reals^d$  \label{alg:sc-rcd_constraint_step}
        \State $\vx_{\mathcal{J}} \leftarrow \vx_{\mathcal{J}} - \valpha$, $\vx_{\mathcal{S}} \leftarrow \vx_{\mathcal{S}} + \vbeta$  \Comment{Update coordinates of iterate in $\mathcal{J}$ and $\mathcal{S}$}
        \State $\vr \leftarrow \vr - \bigl( \mA_{:, \mathcal{J}} \valpha - \mF ((\mF_{\mathcal{J},:})^{\tran} \valpha) \bigr)$  \Comment{Update residual vector} \label{alg:sc-rcd_residual}
    \EndFor
\end{algorithmic}
\end{algorithm}

See Algorithm~\ref{alg:sc-rcd} for pseudocode for an implementation of SC-RCD. The comments provide pointers to parts of Section~\ref{sec:sc-rcd_implementation}, where more details on efficient implementation (such as taking advantage of the lower triangular structure of $\mF_{\mathcal{S},:}$) and computational costs are discussed.

While the updates in~\eqref{eq:sc-rcd_update_intro} have been described in terms of the residual matrix $\mA^{\circ}$, we emphasize that neither $\mA$ nor $\mA^{\circ}$ actually need to be stored in memory.
We only assume that the auxiliary matrix $\mC \in \reals^{d \times n}$ is precomputed and stored, which is feasible if $d \ll n$.
Each iteration of SC-RCD only needs to access the entries $\mA_{:, \mathcal{J}} \in \reals^{n \times \ell}$ in the $\ell$ columns indexed by the sampled subset $\mathcal{J}$.
Furthermore, each iteration can be performed in $O((\ell + d) n)$ arithmetic operations, provided that the block size $\ell$ is at most $O(\sqrt{n})$.

\paragraph{SC-RCD convergence analysis.}
Conditional on the quality of the low-rank approximation output by the RPCholesky algorithm, we show that the SC-RCD method converges linearly in expectation with a rate that depends on the eigenvalues of $\mA$ without the largest spectral outliers:

\begin{theorem} \label{thm:sc-rcd_conv_main}
Let $\mA \vx = \vb$ be a consistent linear system with psd $\mA \in \reals^{n \times n}$ and solution $\vx^*$. Let the eigenvalues of $\mA$ be $\lambda_1(\mA) \geq \dots \geq \lambda_n(\mA) \geq 0$ with smallest non-zero eigenvalue $\lambda_{\mathrm{min}}^+(\mA)$.
Let $\mA \langle \mathcal{S} \rangle$ be a rank-$d$ Nystr\"{o}m approximation computed using the RPCholesky algorithm, and $\{ \vx^k \}_{k \geq 0}$ denote the corresponding sequence of SC-RCD iterates with block size $\ell \geq 1$.
Fix an integer $r \geq 0$ and real $\delta > 0$, $\rho \in (0, 1)$. If the approximation rank $d$ satisfies
\begin{equation} \label{eq:sc-rcd_conv_main_d_cond}
    d \geq \frac{r}{\delta} + r \log\left( \frac{1}{\delta \eta_r} \right), \quad \text{where} \quad
    \eta_r := \frac{\sum_{i > r} \lambda_i(\mA)}{\sum_{i=1}^n \lambda_i(\mA)},
\end{equation}
then, with probability at least $1 - \rho$, the residual matrix $\mA^{\circ} = \mA - \mA \langle \mathcal{S} \rangle$ satisfies
\begin{equation} \label{eq:sc-rcd_conv_main_eventE}
    \tr(\mA^{\circ}) \leq \rho^{-1} (1 + \delta) \sum_{i > r} \lambda_i(\mA).
\end{equation}
Furthermore, conditional on the event that~\eqref{eq:sc-rcd_conv_main_eventE} holds, denoted by $\mathcal{E}$, the expected relative error in the $\mA$-norm satisfies
\begin{equation} \label{eq:sc-rcd_conv_main_rate}
    \E\left[ \norm{\vx^k - \vx^*}_{\mA}^2 \mid \mathcal{E} \right] \leq \left( 1 - \frac{\lambda^+_{\mathrm{min}}(\mA)}{\rho^{-1} (1 + \delta) \sum_{i > r} \lambda_i(\mA)} \right)^{k \ell} \cdot \norm{\vx^0 - \vx^*}_{\mA}^2.
\end{equation}
\end{theorem}

The proof of Theorem~\ref{thm:sc-rcd_conv_main} is given in Section~\ref{sec:sc-rcd}.
More generally, we show that if the subspace constraint is defined by an arbitrary Nystr\"{o}m approximation $\mA \langle \mathcal{S} \rangle$, not necessarily computed using RPCholesky, then the SC-RCD algorithm converges linearly with a rate that depends on the spectrum of the residual matrix $\mA^{\circ} = \mA - \mA \langle \mathcal{S} \rangle$ (see Theorem~\ref{thm:sc-rcd_convrate}).

\begin{remark}[Randomized low-rank approximation] \label{rmk:boosting_prob}
The probability of the event $\mathcal{E}$ in~\eqref{eq:sc-rcd_conv_main_eventE}, which ensures the quality of the randomized low-rank approximation for the iterative phase, can be boosted by the standard trick of running RPCholesky $T \geq 1$ times and choosing the best approximation with the smallest trace-norm. Then, by applying~\eqref{eq:sc-rcd_conv_main_eventE} to each run with probability parameter $\rho = 1/2$, say, we conclude that the best approximation fails to satisfy the same bound in $\mathcal{E}$ with probability at most $2^{-T}$.
Note that this procedure is very easy to run in parallel.
\end{remark}

\begin{remark}[Approximation rank] \label{rmk:on_lowrank_approx}
A technical caveat is the dependence of the approximation rank $d$ on $\eta_r$, the relative trace-norm error of the best rank-$r$ approximation, in~\eqref{eq:sc-rcd_conv_main_d_cond}. Theoretically, it is possible to compute a randomized rank-$d$ Nystr\"{o}m approximation $\widehat{\mA}$ such that
\[
    \E[\tr(\mA - \widehat{\mA})] \leq (1 + \delta) \sum_{i > r} \lambda_i(\mA)
    \quad \text{as long as } d \geq \frac{r}{\delta} + r - 1,
\]
instead of~\eqref{eq:sc-rcd_conv_main_d_cond}, by using a more computationally expensive method for sampling the pivots $\mathcal{S}$ known as \emph{determinantal point process} (DPP) \emph{sampling} (see~\cite[Table~2]{ChenEtAl2024} and the references therein).
Using Markov's inequality to define the event $\mathcal{E}$ in~\eqref{eq:sc-rcd_conv_main_eventE}, the same convergence rate~\eqref{eq:sc-rcd_conv_main_rate} holds with such an approximation.
\end{remark}

\begin{remark}[Block size and sampling] \label{rmk:sc-rcd_blocksize}
Choosing a larger block size $\ell$ for the SC-RCD method means that a heavier $\ell \times \ell$ inner linear system has to be solved (directly or iteratively) in each iteration. This increased cost may be offset by two benefits.
The first is that modern computational architectures can realize computational benefits from using block computations due to parallelism, caching, etc.\ (e.g., see~\cite[App.\ A]{EpperlyEtAl2024rejection}).
The second is that larger blocks results in more progress to be made with each projection; this is reflected by the bound in Theorem~\ref{thm:sc-rcd_conv_main}, which shows that the iteration complexity improves at least linearly in the block size $\ell$. Later, we also show that the entries in each block can be sampled without replacement to obtain a guarantee that is at least as good (see Remark~\ref{rmk:sampling_without_replacement}).

However, we expect (and empirically observe) the actual improvement from projections onto bigger blocks to be even better due to the implicit ``low-rank approximation effect'': e.g., rigorous results showing more precise rates depending on the entire spectrum are known for blocks obtained using DPP sampling~\cite{mutny2020block, rodomanov2020block, XiangXieZhang2025} or subgaussian sketches~\cite{derezinski2024sharp}.
Proving tighter bounds with computationally cheap schemes is more challenging. In a recent line of work, it has been shown that DPP sampling can be approximated by uniform sampling if the input matrix satisfies some incoherence properties~\cite{DerezinskiYang2024, DerezinskiEtAl2024fine, DerezinskiEtAl2025beyond}. When the entries in each block are uniformly sampled in SC-RCD, an analogous result as Theorem~\ref{thm:sc-rcd_conv_main} also holds (see Proposition~\ref{prop:sc-rcd_convrate_unif}).
\end{remark}

\paragraph{SC-RCD complexity.}
Theorem~\ref{thm:sc-rcd_conv_main} shows that the convergence of SC-RCD depends on the approximation rank $d$, the block size $\ell$, and on the \emph{normalized tail condition number} of $\mA$, defined as 
\begin{equation} \label{eq:normalized_demmel_condnum}
    \bar{\kappa}_{r}(\mA) := \frac{1}{n - r} \frac{\sum_{i > r} \lambda_i(\mA)}{\lambda^+_{\mathrm{min}}(\mA)}, \quad \text{ for $r < \mathrm{rank}(\mA)$.}
\end{equation}
Note that $\bar\kappa_r(\mA)$ lower bounds the classical condition number of $\mA - \mA_r$, where $\mA_r$ is the best rank-$r$ approximation of $\mA$ in the Frobenius norm; i.e., $\lambda_{r+1}(\mA) / \lambda^+_{\mathrm{min}}(\mA)$.
When implemented as Algorithm~\ref{alg:sc-rcd}, the SC-RCD method satisfies the following complexity result in terms of the size of the system $n$ and the parameters $d$, $\ell$, and $\bar{\kappa}_{r}(\mA)$, which is proved in Section~\ref{sec:sc-rcd_implementation}.

\begin{theorem} \label{thm:sc-rcd_flat}
Let $\mA \vx = \vb$ be a consistent linear system with psd $\mA \in \reals^{n \times n}$ and solution $\vx^*$. The SC-RCD method (Algorithm~\ref{alg:sc-rcd}) with approximation rank $d \geq 0$ and block size $\ell \geq 1$, combined with the boosting procedure described in Remark~\ref{rmk:boosting_prob}, computes an approximate solution $\vx^k$ that satisfies $\E \norm{\vx^k - \vx^*}_{\mA}^2 \leq \epsilon \cdot \norm{\vx^0 - \vx^*}_{\mA}^2$ after $k = \lceil 4 (n / \ell) \bar{\kappa}_r(\mathbf{A}) \log(2/\epsilon) \rceil$ iterations.
In total,
\begin{equation} \label{eq:thm:sc-rcd_flat_complexity}
    O\bigl( n d^2 \log(1/\epsilon) \bigr) + O\left( \biggl( \frac{n^2(d + \ell)}{\ell} + n \ell^2 + n \ell d \biggr) \cdot \bar{\kappa}_r(\mA) \log(1/\epsilon) \right) \quad \text{operations}
\end{equation}
are required, where $r$ is the largest integer satisfying $d \geq r + r \log\left( 1/\eta_r \right)$ with $\eta_r$ defined as in~\eqref{eq:sc-rcd_conv_main_d_cond}. 
Moreover, $O(nd \log(1/\epsilon)) + O(n^2 \cdot \bar{\kappa}_r(\mA) \log(1/\epsilon))$ entry evaluations of $\mA$, and $O(n(d + \ell))$ storage are required.
Here, big $O$ notation is used to hide absolute constants.
\end{theorem}

In particular, if we choose $\ell = d$, then~\eqref{eq:thm:sc-rcd_flat_complexity} simplifies to
\begin{equation} \label{eq:thm:sc-rcd_flat_complexity_simp}
    O\bigl( (n^2 + n d^2) \cdot \bar{\kappa}_r(\mA) \log(1/\epsilon) \bigr)
    \quad \text{operations}.
\end{equation}
We see that the SC-RCD method takes advantage of the flat-tailed structure of the spectrum of $\mA$ (i.e., when $\bar{\kappa}_{r}(\mA)$ is well-bounded), which can appear in practice due to explicit regularization or the effects of noise in the data. 

\begin{remark} \label{rem:krr}
To illustrate the claim of Theorem~\ref{thm:sc-rcd_flat}, let us suppose that $\mA$ is not extremely ill-conditioned. Specifically, suppose that there exists an absolute constant $\tau > 0$ such that
\begin{equation} \label{eq:thm:sc-rcd_flat_conds}
    \frac{\sum_{i=1}^n \lambda_i(\mA)}{\lambda^+_{\mathrm{min}}(\mA)} \leq O(n^{\tau}).
\end{equation}
Then, $r$ can be taken as large as $r = O(d / \log n)$ in the complexity bounds in Theorem~\ref{thm:sc-rcd_flat}.
An important setting in which~\eqref{eq:thm:sc-rcd_flat_conds} is often satisfied is kernel ridge regression, where $\mA = \mK + \lambda \mI$, the kernel matrix $\mK$ has unit diagonals ($\mK_{ii} = 1$ for all $i$), and the regularization parameter $\lambda$ is not too small (i.e., $\lambda \geq Cn^{-\tau}$).
In addition, if $\mK$ also exhibits rapid spectral decay, then $\mA$ would be expected to have a well-bounded normalized tail condition number (governed by $\lambda$).
In particular, if we assume that $\bar{\kappa}_r(\mA) = O(1)$ for some $r = O(\sqrt{n} / \log n)$, then, Theorem~\ref{thm:sc-rcd_flat} implies that SC-RCD with $d = O(\sqrt{n})$ and $\ell = O(\sqrt{n})$ can solve the KRR problem to $\epsilon$-relative error using $O(n^2 \log(1/\epsilon))$ arithmetic operations. Note that this is optimal in terms of $n$ for solving a dense $n \times n$ linear system.
\end{remark}

\paragraph{\textit{II. A subspace-constrained sketch-and-project framework}.}
To analyze the SC-RCD method, we consider the algorithm as an instance of a more general \emph{subspace-constrained sketch-and-project} framework, which we develop in Section~\ref{sec:sc-sap} and may be of independent interest.
The \emph{sketch-and-project method}~\cite{GowerRichtarik2015} encompasses a class of iterative algorithms for solving linear systems $\mA \vx = \vb$ with $\mA \in \reals^{m \times n}$ (not necessarily psd), including the randomized Kaczmarz~\cite{StrohmerVershynin2009}, randomized coordinate descent~\cite{LeventhalLewis2010}, and randomized Newton~\cite{GKLR2019} methods.

In each iteration of the standard sketch-and-project algorithm, a sketching matrix $\mS \in \reals^{\ell \times m}$ is drawn independently from an input distribution $\mathcal{D}$, and the current iterate $\vx^k \in \reals^n$ is projected onto the sketched linear system $\mS \mA \vx = \mS \vb$ with respect to the norm $\norm{\vx}_{\mB} = \sqrt{\vx^{\tran} \mB \vx}$ induced by a positive definite matrix parameter $\mB \in \reals^{n \times n}$:
\begin{equation} \label{eq:sketch_and_project_intro}
    \vx^{k+1} = \argmin_{\vx \in \reals^n} \norm{\vx - \vx^k}_{\mB} \quad \text{such that} \quad \mS \mA \vx = \mS \vb.
\end{equation}

We propose to constrain the dynamics of the iterates from~\eqref{eq:sketch_and_project_intro} within a particular (affine) subspace parameterized by a matrix parameter $\mQ \in \reals^{d \times m}$:
\[
    \vx^k \in \{ \vx \in \reals^n: \mQ \mA \vx = \mQ \vb \} \quad \text{for all } k \geq 0.
\]
In Section~\ref{sec:sc-sap}, we show that the theory for the subspace-constrained sketch-and-project method parallels the standard sketch-and-project method. The main difference is the emergence of a projector $\mP_{\mB}$ onto the subspace $\nullspace(\mQ \mA)$ that acts to constrain the iterates. We prove that the convergence rate depends on a projected version of the original matrix, $\mA \mP_{\mB}$, instead of $\mA$, which shows that the subspace constraint effectively acts as a preconditioner and can speed up the convergence in various general cases (see Theorem~\ref{thm:convergence_rate} and Remark~\ref{rmk:update_B}).

If the matrix $\mA \in \reals^{n \times n}$ is positive definite, then the SC-RCD method can be derived as a special case of this framework by choosing $\mB = \mA$; extremely sparse sketching matrices of the form $\mS = \ve_{\mathcal{J}}^{\tran}$, defined as in~\eqref{eq:block_rcd_update}; and a subspace constraint defined by the matrix $\mQ = \ve_{\mathcal{S}}^{\tran}$, which can be identified with the set of pivots $\mathcal{S} \subseteq [n]$ corresponding to a Nystr\"{o}m approximation $\mA \langle \mathcal{S} \rangle$ of $\mA$.

\paragraph{\textit{III. Randomized block Kaczmarz convergence rate}.}
As a part of our analysis of the sketch-and-project framework---which is not specific to the subspace-constrained version---we show that the iteration complexity improves at least linearly in the block size $\ell$ when the sketching matrices $\mS \in \reals^{\ell \times m}$ consist of $\ell$ independent and identically distributed rows (in Proposition~\ref{prop:convergence_rate_block}).

Previously, bounds for block sketch-and-project methods have only been obtained for more special sketching matrices that are dense~\cite{DeLiLiMa2020, derezinski2024sharp} or based on more computationally expensive DPP sampling schemes~\cite{mutny2020block, DerezinskiEtAl2024fine}. With more randomness, the rates obtained are generally sharper. However, Proposition~\ref{prop:convergence_rate_block} applies more generically, such as when standard coordinate basis vectors are sampled using a simple scheme.

Concretely, this result applies to the randomized block Kaczmarz method~\cite{Elfving1980, NeedellTropp2014} when the blocks consist of i.i.d.\ rows of $\mA$, sampled with probability proportional to the squared row norms, which is a direct block generalization of the classical randomized Kaczmarz method~\cite{StrohmerVershynin2009}.

\begin{proposition} \label{prop:block_rk_rate}
Consider solving the consistent linear system $\mA \vx = \vb$ with $\mA \in \reals^{m \times n}$ and solution $\vx^*$. Let $\{ \vx^k \}_{k \geq 0}$ be the iterates of the randomized block Kaczmarz method, defined by
\[
    \vx^{k+1} = \vx^k - (\mA_{\mathcal{J},:})^{\dagger} (\mA_{\mathcal{J},:} \vx^k - \vb_{\mathcal{J}}),
\]
where the block $\mathcal{J} = \{ j_1, \ldots, j_\ell \}$ in each iteration consists of $\ell$ rows of $\mA$, independently sampled from the distribution $\{ \norm{\mA_{j,:}}_2^2 / \norm{\mA}_F^2 \}_{j=1}^m$. If $\sigma^+_{\mathrm{min}}(\mA)$ denotes the smallest non-zero singular value of $\mA$, then
\[
    \E \norm{\vx^k - \vx^*}_2^2 \leq \left( 1 - \frac{\sigma_{\mathrm{min}}^+(\mA)^2}{\norm{\mA}_F^2} \right)^{\ell k} \cdot \norm{\vx^0 - \vx^*}_2^2.
\]
\end{proposition}

Proposition~\ref{prop:block_rk_rate} is a special case of Corollary~\ref{cor:sc_block_rk_rate} for the subspace-constrained randomized block Kaczmarz method, which appears later in Section~\ref{sec:sc-sap}.
To the best of our knowledge, this bound, which describes the explicit dependence of the rate on the block size, is new even for the standard randomized block Kaczmarz method.

In prior work~\cite{NeedellTropp2014}, a randomized block Kaczmarz method, which additionally requires $\mA$ to be row-normalized, was analyzed based on sampling non-intersecting blocks that partition the rows of the matrix.
The sampling strategy in Proposition~\ref{prop:block_rk_rate} is more general and slightly differs: in particular, it can result in repeated rows in the blocks. However, the convergence rate can only improve if the blocks are sampled without replacement instead, and we refer to Remark~\ref{rmk:sampling_without_replacement}, which appears later in Section~\ref{sec:sc-rcd_implementation}, for more details.

\subsection{Related works}

\paragraph{Coordinate descent and sketch-and-project.}
Important instances of the sketch-and-project framework~\cite{GowerRichtarik2015, RichtarikTakac2020, GowerHaRiSt2018, GowerMoMoNe2021} for solving linear systems include the randomized coordinate descent (RCD)~\cite{LeventhalLewis2010, Nesterov2012}---also known as randomized Gauss--Seidel---and the closely-related randomized Kaczmarz (RK)~\cite{StrohmerVershynin2009} methods.
The RK and RCD methods and their variants have been extensively studied, including~\cite{MaNeedellRamdas2015, HefnyNeRa2017, NeSrWa2016, NeedellTropp2014, SchopferLorenz2019, HaNeReSw2022, EpperlyGoldshlagerWebber2024, rathore2025turbo}.
Notably, the subspace-constrained sketch-and-project framework developed in this paper generalizes a subspace-constrained RK method analyzed in~\cite{LokRebrova2024}. 
Note that coordinate descent algorithms for more general objective functions have also been extensively studied in the optimization literature: e.g., see~\cite{RichtarikTakac2014, RichtarikTakac2016, QuRichtarik2016, TsengYun2010}.

An accelerated sketch-and-project method with Nesterov momentum was analyzed in~\cite{GowerHaRiSt2018}, building on earlier works on accelerated RK~\cite{LiuWright2016acc} and RCD~\cite{Nesterov2012, LeeSidford2013, NesterovStich2017, TuEtAl2017locality}. The accelerated method requires additional tuning parameters: theoretically, with a specific choice of values (which are not easily computable), the accelerated method leads to an improved convergence rate bound. In particular, \cite{TuEtAl2017locality} presents experiments showing that accelerated RCD can outperform RCD and the conjugate gradient method for large-scale KRR problems in machine learning.

A closely related line of works~\cite{DerezinskiYang2024, DerezinskiEtAl2024fine, DerezinskiEtAl2025beyond} analyzes randomized solvers based on sketch-and-project that are also especially effective for solving approximately low-rank systems, building on the insight that sketch-and-project can exploit rapid spectral decay~\cite{derezinski2024sharp}.
Most recently, Derezi{\'n}ski et al.~\cite{DerezinskiEtAl2025beyond} showed that an algorithm based on accelerated RCD called CD++ can exploit large spectral outliers. Specifically, they prove that for any $\tilde{O}(1) \leq r \leq n$ (where $\tilde{O}$ hides polylog factors in $n$), a solution of the psd system $\mA \vx = \vb$ with $\epsilon$-relative error in $\norm{\cdot}_{\mA}$ can be computed using
\begin{equation} \label{eq:cd++_complexity}
    \tilde{O}(nr^2) + \tilde{O}(n^2 \sqrt{\bar{\kappa}_r(\mA)} \log(1/\epsilon))
    \quad \text{operations (\cite[\S 5]{DerezinskiEtAl2025beyond})}.
\end{equation}
The CD++ algorithm achieves~\eqref{eq:cd++_complexity} through the use and analysis of techniques such as adaptive acceleration, approximate regularized projections, randomized Hadamard preconditioning, and block memoization. One of the main takeaways is that the rate of CD++ cannot be outperformed by any solver based on matrix-vector products, such as Krylov subspace-based solvers (\cite[Theorem~3]{DerezinskiEtAl2024fine}). 

The SC-RCD method with $d = O(r \log n)$, $\ell = d$, and $1 \leq r \leq O(\sqrt{n} / \log n)$ has a similar complexity bound as~\eqref{eq:cd++_complexity} for flat-tailed systems with $\bar{\kappa}_r(\mA) = O(1)$ satisfying condition~\eqref{eq:thm:sc-rcd_flat_conds}. Otherwise, CD++ has a better dependence on the normalized tail condition number due to the incorporation of Nesterov's momentum in the algorithm, which could be integrated into SC-RCD as a part of future work.
At the same time, SC-RCD does not require storing the entire input matrix, while the CD++ guarantee~\eqref{eq:cd++_complexity} assumes that $O(n^2)$ memory is available for storing a preprocessed version of the matrix.
Algorithmically, the main difference is that CD++ \emph{implicitly} captures the leading part of the spectrum of $\mA$, whereas it is \emph{explicitly} learned in SC-RCD by constraining the dynamics of RCD.
The SC-RCD method is based on the combination of two fundamental ideas: low-rank matrix approximation and a simple iterative solver. We believe that this allows for flexibility: e.g., the aforementioned innovations in CD++ can be combined with the subspace constraint idea.

\paragraph{Conjugate gradient method.}
Another popular class of iterative solvers is based on Krylov subspaces~\cite{NakatsukasaTropp2024}, which includes the \emph{conjugate gradient method} (CG)~\cite{Saad2003} for solving psd systems. In practice, preconditioning is often crucial for these methods to be effective.
For kernel ridge regression (KRR) specifically, a \emph{preconditioned conjugate gradient} (PCG) method for solving $(\mK + \lambda \mI) \vx = \vy$ using a low-rank Nystr\"{o}m approximation $\widehat{\mK} = \mF \mF^{\tran}$ of $\mK$ computed with RPCholesky was analyzed by D{\'{i}}az et al.~\cite{DiazEtAl2023robust}.
We note that there many other methods, based on dimension reduction, have been proposed to solve large-scale KRR problems: e.g., we refer to~\cite{rathore2025askotch, DiazEtAl2023robust, avron2017kernel, meanti2020kernel, rudi2017falkon, RahimiRecht2007, SmolaBartlett2000, WilliamsSeeger2000} and the references therein.

To give a brief overview of the PCG method of~\cite{DiazEtAl2023robust}, a preconditioner $\mM = \widehat{\mK} + \lambda \mI$ is constructed from an SVD of $\mF$, and the spectrum of $\mM^{-1/2} (\mK + \lambda \mI) \mM^{-1/2}$ determines the resulting rate of convergence.
It is shown that if the approximation rank $d$ is larger than the $\lambda$-tail rank of $\mK$, which is defined as the smallest integer $r$ such that $\sum_{i > r} \lambda_i(\mK) \leq \lambda$, then the preconditioned matrix has constant condition number, meaning that $O(n^2 \log(1/\epsilon))$ operations suffices to compute a solution with $\epsilon$-relative error in $\norm{\cdot}_{\mA}$ (\cite[Theorem~2.2]{DiazEtAl2023robust}).
This is a similar result as Theorem~\ref{thm:sc-rcd_flat} for SC-RCD: however, we note that SC-RCD can be applied to general psd systems.

\paragraph{Low-rank approximation.}
Low-rank matrix approximation is a fundamental idea in numerical linear algebra~\cite{MartinssonTropp2020}. The most relevant forms for SC-RCD are based on (column) Nystr\"{o}m approximation or interpolative decomposition~\cite{ChenEtAl2024, DongEtAl2024, EpperlyEtAl2024rejection}. These are aligned with the standard coordinate basis, which means that the corresponding subspace constraint can be enforced by updating a fixed subset of coordinates.
These algorithms use random, adaptive pivoting, and are accompanied by provable guarantees of quality that are comparable to the best low-rank approximation.
Other simple strategies such as uniform sampling and greedy pivoting may work well in practice, but do not have good error bounds in general.
DPP sampling~\cite{DerezinskiMahoney2021}, which is based on sampling blocks weighted by their squared volumes, offers the best known near-optimal bounds, but are not easily computable.
Recently, a greedy deterministic method based on maximization of a trace-norm is also analyzed in~\cite{FornaceLindsey2024}.

More generally, the subspace constraint for sketch-and-project can be defined using low-rank approximations from \emph{random embeddings} of $\mA$, which mix across all the coordinates to perform dimension reduction (e.g., multiplying by a Gaussian matrix).
These methods are often more robust and lead to a better approximation than those based on sampling. We refer to the survey~\cite[\S 5.4]{TroppWebber2023lowrank} for a comprehensive discussion of Nystr\"{o}m-based adaptations of matrix approximation algorithms such as randomized SVD~\cite{HaMaTr2011}, as well as extensions based on block Krylov iteration, which are more accurate for matrices with slow spectral decay.

\subsection{Organization}

The rest of the paper is structured as follows. Section~\ref{sec:sc-sap} describes a more abstract subspace-constrained sketch-and-project framework for solving general linear systems.
In Section~\ref{sec:sc-rcd}, the analysis is specialized to the SC-RCD method for solving psd linear systems.
Section~\ref{sec:numerical_results} presents numerical experiments demonstrating the performance of SC-RCD for solving psd systems, generated synthetically and coming from kernel ridge regression problems using real-life datasets.
Appendix~\ref{sec:sc-sap_proofs} contains some omitted technical proofs, Appendix~\ref{sec:least_squares} describes an extension of the SC-RCD method for solving least squares problems, and Appendix~\ref{sec:additional_numexp} presents additional numerical experiments.

\section{A general framework for subspace-constrained sketch-and-project} \label{sec:sc-sap}

Consider the goal of solving a consistent system of linear equations $\mA \vx = \vb$, where $\mA \in \reals^{m \times n}$. Since the matrix $\mA$ may not necessarily be square or have full rank in this generality, we aim to find the min-norm solution
\begin{equation} \label{eq:defn_xstar}
    \vx^* := \argmin_{\vx \in \reals^n} \norm{\vx - \vx^0}_{\mB} \quad \text{such that} \quad \mA \vx = \vb,
\end{equation}
where $\vx^0 \in \reals^n$ represents any initial approximate solution and the norm is induced by a psd matrix parameter $\mB \in \reals^{n \times n}$. Observe that $\vx^*$ is the projection of $\vx^0$ onto $\mA \vx = \vb$ with respect to the $\mB$-norm, so equivalently,
\begin{equation} \label{eq:xstar_formula}
    \vx^* = \vx^0 - \mB^{-1} \mA^{\tran} (\mA \mB^{-1} \mA^{\tran})^{\dagger} (\mA \vx^0 - \vb).
\end{equation}

\paragraph{Overview of the standard sketch-and-project method.}
In each iteration of the vanilla sketch-and-project framework~\cite{GowerRichtarik2015}, a sketching matrix $\mS \equiv \mS^k$ is drawn independently from an input distribution $\mathcal{D}$, and the current iterate $\vx^k \in \reals^n$ is projected onto the sketched linear system $\mS \mA \vx = \mS \vb$ with respect to the $\mB$-norm as in~\eqref{eq:sketch_and_project_intro}.
As shown in~\cite{GowerRichtarik2015}, this update can be written in closed form as
\begin{equation} \label{eq:sketch_and_project_update}
    \vx^{k+1} = \vx^k - \mB^{-1} \mA^{\tran} \mS^{\tran} (\mS \mA \mB^{-1} \mA^{\tran} \mS^{\tran})^{\dagger} \mS(\mA \vx^k - \vb),
\end{equation}
or expressed as a fixed point iteration in terms of the solution $\vx^*$ from~\eqref{eq:defn_xstar} as
\begin{equation} \label{eq:sketch_and_project_fixedpoint}
    \vx^{k+1} - \vx^* = (\mI - \widetilde{\mZ}) (\vx^k - \vx^*),
\end{equation}
where $\widetilde{\mZ}$ is the orthogonal projector onto $\range(\mB^{-1/2} \mA^{\tran} \mS^{\tran})$: that is, 
\begin{equation} \label{eq:sketch_and_project_projector}
    \widetilde{\mZ} := \mB^{-1/2} \mA^{\tran} \mS^{\tran} (\mS \mA \mB^{-1} \mA^{\tran} \mS^{\tran})^{\dagger} \mS \mA \mB^{-1/2}.
\end{equation}
The convergence rate of the sketch-and-project algorithm depends on the eigenvalue spectrum of the expected projection matrix $\E[\widetilde{\mZ}]$, where the expectation is taken over the random sketching matrix $\mS \sim \mathcal{D}$ (\cite[Theorem~4.6]{GowerRichtarik2015}).
For a wide class of sketching matrices (including $\mS$ with independent subgaussian entries as well as sparse counterparts), it has been shown that the convergence rate of sketch-and-project precisely depends on the entire singular value spectrum of $\mA$~\cite{DeLiLiMa2020, derezinski2024sharp}.

\subsection{The subspace-constrained sketch-and-project framework}

Let $\mQ \in \reals^{d \times m}$ with $d < m$ be an arbitrary matrix parameter defining the (affine) subspace $\{ \vx \in \reals^n: \mQ \mA \vx = \mQ \vb \}$.
We formulate the \emph{subspace-constrained sketch-and-project method} as the process starting from an initial iterate $\vx^0 \in \reals^n$ solving $\mQ \mA \vx^0 = \mQ \vb$, and with subsequent iterates given by
\begin{equation} \label{eq:sc_sap}
    \vx^{k+1} = \argmin_{\vx \in \reals^n} \norm{\vx - \vx^k}_{\mB} \quad \text{such that} \quad \mS \mA \vx = \mS \vb,\; \mQ \mA \vx = \mQ \vb,
\end{equation}
where the sketching matrices $\mS \equiv \mS^k$ in each iteration are drawn independently from an input distribution $\mathcal{D}$.
This is equivalent to the sketch-and-project method with the iterates $\vx^k$ confined within the subspace $\{ \vx \in \reals^n: \mQ \mA \vx = \mQ \vb \}$.

The matrix $\mQ$ could be chosen deterministically and consist of a selection of $d$ rows of the system of linear equations. It could also be a random embedding of the linear system.
Informally, constraining the iterates to a subspace parametrized by $\mQ$ is computationally beneficial when
\begin{enumerate}[label=\arabic*., leftmargin=2em]
    \item we can efficiently solve the smaller system $\mQ \mA \vx = \mQ \vb$ to obtain an initial iterate $\vx^0$; and
    \item we can efficiently project onto the nullspace of $\mQ \mA \mB^{-1/2}$ in each iteration.
\end{enumerate}

In the rest of this section, our goal is to develop general theory for the subspace-constrained sketch-and-project method~\eqref{eq:sc_sap}. Specifically, we will derive closed-form expressions for the iterations and estimates for the convergence rate, and demonstrate the theoretical advantage of having a subspace constraint.

\subsection{Update rule and comparison with unconstrained sketch-and-project}

Our first goal is to show that the iterations of subspace-constrained sketch-and-project admit the following closed-form expressions, analogous to~\eqref{eq:sketch_and_project_update} and~\eqref{eq:sketch_and_project_fixedpoint} for the unconstrained method.

\begin{lemma}[Closed-form updates] \label{lem:update}
Let $\mP$ and $\mZ$ be the orthogonal projection matrices onto $\nullspace(\mQ \mA \mB^{-1/2})$ and $\range(\mP \mB^{-1/2} \mA^{\tran} \mS ^{\tran})$ respectively; that is,
\begin{align}\label{z-def}
    \mP &:= \mI - (\mQ \mA \mB^{-1/2})^{\dagger} \mQ \mA \mB^{-1/2}, \nonumber\\
    \mZ &:= \mP \mB^{-1/2} \mA^{\tran} \mS^{\tran} (\mS \mA \mB^{-1/2} \mP \mB^{-1/2} \mA^{\tran} \mS^{\tran})^{\dagger} \mS \mA \mB^{-1/2} \mP.
\end{align}
If $\vx^0$ is any vector satisfying $\mQ \mA \vx^0 = \mQ \vb$, then the subspace-constrained sketch-and-project iterates $\{ \vx^k \}_{k \geq 0}$ from~\eqref{eq:sc_sap} satisfy the following:
\begin{equation} \label{eq:lem_update_1}
    \vx^{k+1} = \vx^k - \mB^{-1/2} \mP \mB^{-1/2} \mA^{\tran} \mS^{\tran} (\mS \mA \mB^{-1/2} \mP \mB^{-1/2} \mA^{\tran} \mS^{\tran})^{\dagger} \mS(\mA \vx^k - \vb).
\end{equation}
Furthermore, if $\vx^*$ is the min-norm solution as in~\eqref{eq:defn_xstar}, then
\begin{equation} \label{eq:lem_update_2}
    \mB^{1/2} (\vx^{k+1} - \vx^*) = (\mI - \mZ) \mB^{1/2}(\vx^k - \vx^*).
\end{equation}
\end{lemma}

The proof of Lemma~\ref{lem:update}, which is technical, can be found in Appendix~\ref{sec:sc-sap_proofs}.
As a quick sanity check, observe that with $\mQ = \mzero$ and $\mP = \mI$ in Lemma~\ref{lem:update}, which corresponds to having no subspace constraint, we recover the sketch-and-project updates~\eqref{eq:sketch_and_project_update} and~\eqref{eq:sketch_and_project_fixedpoint}.

A consequence of the closed form update formulas is the following auxiliary lemma, whose proof can also be found in Appendix~\ref{sec:sc-sap_proofs}.

\begin{lemma}[Invariant subspace property] \label{lem:error_invariance}
Let $\vx^*$ be the min-norm solution to $\mA \vx = \vb$ defined in~\eqref{eq:defn_xstar} and $\vx^k$ be the $k$th iterate from the subspace-constrained sketch-and-project method \eqref{eq:sc_sap}. Then for all $k \geq 0$, $\mB^{1/2}(\vx^k - \vx^*) \in \range(\mP \mB^{-1/2} \mA^{\tran})$.
\end{lemma}

The next result gives a formula for the decrease in error for each iteration of subspace-constrained sketch-and-project, and justifies that it is always at least as large as the error decrease from the corresponding unconstrained update.

\begin{lemma}[Error decrease] \label{lem:error_decrease}
With the same notation as in Lemma~\ref{lem:update}, the decrease in error in each iteration of subspace-constrained sketch-and-project is given by
\begin{equation} \label{eq:error_decrease_1}
    \norm{\mB^{1/2}(\vx^{k+1} - \vx^*)}_2^2 = \norm{\mB^{1/2}(\vx^k - \vx^*)}_2^2 - \norm{\mZ \mB^{1/2} (\vx^k - \vx^*)}_2^2.
\end{equation}
Furthermore, if $\widetilde{\mZ}$ is the corresponding projection matrix for the unconstrained sketch-and-project method from~\eqref{eq:sketch_and_project_projector} with the same sketching matrix $\mS$, then the error decrease with the subspace constraint is not smaller than the error decrease in the unconstrained case:
\begin{equation} \label{eq:comparison_vanilla_1}
    \norm{\mZ \mB^{1/2} (\vx^k - \vx^*)}_2 \geq \norm{\widetilde{\mZ} \mB^{1/2} (\vx^k - \vx^*)}_2.
\end{equation}
\end{lemma}

\begin{proof}
Since $\mZ$ is an orthogonal projector, the per-iteration error decrease~\eqref{eq:error_decrease_1} in terms of the Euclidean norm follows from the fixed point equation~\eqref{eq:lem_update_2} in Lemma~\ref{lem:update} and Pythagoras' theorem. 

Next, observe that showing~\eqref{eq:comparison_vanilla_1} is equivalent to showing
\begin{equation} \label{eq:comparison_vanilla_2}
    \vxi^{\tran} (\mS \mA \mB^{-1/2} \mP \mB^{-1/2} \mA^{\tran} \mS^{\tran})^{\dagger} \vxi \geq \vxi^{\tran} (\mS \mA \mB^{-1} \mA^{\tran} \mS^{\tran})^{\dagger} \vxi
\end{equation}
for all $\vxi = \mS \mA(\vx^k - \vx^*) = \mS \mA \mB^{-1/2} \mP \mB^{1/2}(\vx^k - \vx^*)$, recalling that $\mB^{1/2}(\vx^k - \vx^*) \in \range(\mP)$ by Lemma~\ref{lem:error_invariance}.
Since $\mP \preceq \mI$ in the psd order, we immediately have that
\begin{equation} \label{eq:comparison_vanilla_3}
    \mS \mA \mB^{-1/2} \mP \mB^{-1/2} \mA^{\tran} \mS^{\tran} \preceq \mS \mA \mB^{-1} \mA^{\tran} \mS^{\tran}.
\end{equation}
Now, let $\mathcal{L} := \nullspace(\mP \mB^{-1/2} \mA^{\tran} \mS^{\tran})$. Observe that $\vxi \in \range(\mS \mA \mB^{-1/2} \mP) = \mathcal{L}^{\perp}$ and
\[
    \nullspace(\mS \mA \mB^{-1} \mA^{\tran} \mS^{\tran}) = \nullspace(\mB^{-1/2} \mA^{\tran} \mS^{\tran}) \subseteq \mathcal{L} = \nullspace(\mS \mA \mB^{-1/2} \mP \mB^{-1/2} \mA^{\tran} \mS^{\tran}),
\]
and so
$
    \left( \left. \mS \mA \mB^{-1/2} \mP \mB^{-1/2} \mA^{\tran} \mS^{\tran} \right|_{\mathcal{L}^{\perp}} \right)^{\dagger} \succeq \left( \left. \mS \mA \mB^{-1} \mA^{\tran} \mS^{\tran} \right|_{\mathcal{L}^{\perp}} \right)^{\dagger},
$ 
which implies~\eqref{eq:comparison_vanilla_2}.
\end{proof}

We will now prove that the subspace-constrained sketch-and-project method converges (in mean squared error and the $\mB$-norm) with a rate that depends on the eigenvalue spectrum of the expected projection matrix $\E[\mZ]$, where $\mZ$ is defined in~\eqref{z-def} and the expectation is taken over the distribution of the random sketching matrix $\mS$.

\begin{theorem} \label{thm:convergence_rate}
Suppose that the same notation as Lemma~\ref{lem:update} is used. Assume that the following exactness condition holds:
\begin{equation} \label{eq:exactness}
    \nullspace(\E[\mZ]) = \nullspace(\mA \mB^{-1/2} \mP).
\end{equation}
Then the subspace-constrained sketch-and-project iterates $\{ \vx^k \}_{k \geq 0}$ satisfy
\[
    \E \norm{\vx^k - \vx^*}_{\mB}^2 \leq \bigl( 1 - \lambda^+_{\mathrm{min}}(\E[\mZ]) \bigr)^k \cdot \norm{\vx^0 - \vx^*}_{\mB}^2,
\]
where
\[
    \lambda^+_{\mathrm{min}}(\E[\mZ]) = \min_{\substack{\vx \in \range(\mP \mB^{-1/2} \mA^{\tran})\\ \norm{\vx}_2 = 1}} \vx^{\tran} (\E[\mZ]) \vx.
\]
\end{theorem}

\begin{proof}[Proof of Theorem~\ref{thm:convergence_rate}]
By considering the decrease in error in the $(k+1)$st iteration from Lemma~\ref{lem:error_decrease}, taking the expectation conditional on all the randomness up to the $k$th iteration, which we denote by $\E_k$, and using the linearity of expectation, we obtain
\[
    \E_k \norm{\vx^{k+1} - \vx^*}_{\mB}^2 = \norm{\vx^k - \vx^*}_{\mB}^2 - (\vx^k - \vx^*)^{\tran} \mB^{1/2} (\E[\mZ]) \mB^{1/2} (\vx^k - \vx^*).
\]
By Lemma~\ref{lem:error_invariance}, the vectors $\mB^{1/2}(\vx^k - \vx^*) \in \range(\mP \mB^{-1/2} \mA^{\tran}) = \nullspace(\mA \mB^{-1/2} \mP)^{\perp}$ for all $k \geq 0$. Since $\nullspace(\E[\mZ]) = \nullspace(\mA \mB^{-1/2} \mP)$ under the assumption~\eqref{eq:exactness}, we may expand $\mB^{1/2}(\vx^k - \vx^*)$ in the orthonormal basis of eigenvectors of $\E[\mZ]$ corresponding to positive eigenvalues only, and so
\[
    (\vx^k - \vx^*)^{\tran} \mB^{1/2} (\E[\mZ]) \mB^{1/2} (\vx^k - \vx^*) \geq \lambda_{\mathrm{min}}^+(\E[\mZ]) \cdot \norm{\mB^{1/2}(\vx^k - \vx^*)}_2^2.
\]
Hence, the decrease in error in the $(k+1)$st iteration can be bounded by
\[
    \E_k \norm{\vx^{k+1} - \vx^*}_{\mB}^2 \leq \bigl( 1 - \lambda_{\mathrm{min}}^+(\E[\mZ]) \bigr) \cdot \norm{\vx^k - \vx^*}_{\mB}^2.
\]
We conclude by iterating and using the tower rule for conditional expectations.
\end{proof}

\begin{remark}[Exactness condition]
The assumption~\eqref{eq:exactness} is a technical condition to ensure that the convergence rate is (strictly) positive. It is similar to an exactness condition from the sketch-and-project literature~\cite{RichtarikTakac2020, GowerMoMoNe2021, GowerRichtarik2015ascent}, which relaxes stronger requirements such as $\mA$ having full column rank, and holds for most practical sketching techniques. 
Intuitively, the exactness condition fails to hold if the distribution of sketching matrices $\mS \sim \mathcal{D}$ does not cover the entire space. For example, if the updates are sampled from the same low-rank subspace, then $\vx^0 - \vx^*$ may have components that are unable to be resolved.
\end{remark}

\begin{remark}[$\mB$-inner product geometry and oblique projections] \label{rmk:update_B}
The natural geometry of the sketch-and-project method is defined by the $\mB$-inner product $\innprod{\vx}{\vy}_{\mB} := \vx^{\tran} \mB \vy$ corresponding to the positive definite $\mB \in \reals^{n \times n}$.
By expressing the projectors $\mP$ and $\mZ$ from Lemma~\ref{lem:update} in terms of this geometry via a similarity transformation, the update formulas for the subspace-constrained sketch-and-project method admit the following equivalent, more natural expressions.
Specifically, let
\begin{equation} \label{eq:oblique_projectors}
    \mP_{\mB} := \mB^{-1/2} \mP \mB^{1/2}
    \quad \text{and} \quad
    \mZ_{\mB} := \mB^{-1/2} \mZ \mB^{1/2}
\end{equation}
be the \emph{oblique} projection matrices onto $\nullspace(\mQ \mA)$ and $\range(\mB^{-1} \mP_B^{\tran} \mA^{\tran} \mS^{\tran})$ with respect to the $\mB$-norm respectively.\footnote{Note that if $\mPi$ is the orthogonal projector onto some subspace $\mathcal{M}$, then $\mPi_{\mB} = \mB^{-1/2} \mPi \mB^{1/2}$ is the orthogonal projector onto $\mathcal{M}$ with respect to the $\mB$-norm (i.e., $\mPi_{\mB} \vx = \argmin_{\vy \in \mathcal{M}} \norm{\vx - \vy}_{\mB}$).}
Note that $\mP_{\mB} \mB^{-1} \mP_{\mB}^{\tran} = \mP_{\mB} \mB^{-1} = \mB^{-1} \mP_{\mB}^{\tran}$ and $\mP_{\mB} (\vx^k - \vx^*) = \vx^k - \vx^*$.
Then~\eqref{eq:lem_update_1} and~\eqref{eq:lem_update_2} in Lemma~\ref{lem:update} can be written as
\begin{align}
    \vx^{k+1} &= \vx^k - \mB^{-1} \mP_{\mB}^{\tran} \mA^{\tran} \mS^{\tran} (\mS \mA \mP_{\mB} \mB^{-1} \mP_{\mB}^{\tran} \mA^{\tran} \mS^{\tran})^{\dagger} \mS(\mA \vx^k - \vb), \\
    \vx^{k+1} - \vx^* &= (\mI - \mZ_{\mB})(\vx^k - \vx^*).
\end{align}
These expressions are almost the same as~\eqref{eq:sketch_and_project_update} and~\eqref{eq:sketch_and_project_fixedpoint} for the unconstrained sketch-and-project method using a ``new matrix'' $\mA \mP_{\mB}$ in place of $\mA$, and the same sketched residuals $\mS(\mA \vx^k - \vb)$.
Similarly, the error decrease in Lemma~\ref{lem:error_decrease} can be written in terms of the $\mB$-norm as
\[
    \norm{\vx^{k+1} - \vx^*}_{\mB}^2 = \norm{\vx^k - \vx^*}_{\mB}^2 - \norm{\mZ_{\mB} (\vx^k - \vx^*)}_{\mB}^2.
\]
\end{remark}

\begin{remark}[Connections with the nullspace method]
Observe that to solve $\mA \vx = \vb$, we can first solve $\mQ \mA \vx^0 = \mQ \vb$ for $\vx^0$. Then, given $\vx^0$, we can (approximately) solve $\mA \mP_\mB \vw \approx \vb - \mA \vx^0$ for $\vw \in \range(\mP_\mB) = \nullspace(\mQ \mA)$. The overall solution can then be constructed as $\vx^1 = \vx^0 + \vw$, which continues to satisfy the constraints $\mQ \mA \vx^1 = \mQ \vb$.
In our setting, we form $\vx^1$ by solving the sketched system $\mS \mA \mP_\mB \vw = \mS(\vb - \mA \vx^0)$, and iterate this procedure to obtain the iterates $\{ \vx^k \}_{k \geq 0}$.

This setup resembles the nullspace approach for solving least squares problems with linear equality constraints (e.g., see~\cite{ScottTuma2022}), and illuminates the observation in Remark~\ref{rmk:update_B} that the subspace-constrained updates parallel the unconstrained updates using a projected version of the matrix, $\mA \mP_{\mB}$, in place of $\mA$.
In our context, the main difference is that the subspace constraint $\mQ \mA \vx = \mQ \vb$ is not specified by the problem a priori, but generated algorithmically from the data. For our purposes, \emph{the subspace constraint effectively acts as a preconditioner}: ideally, the properties of $\mA \mP_{\mB}$ should improve the downstream iterative process in a suitable way.
For instance, for the SC-RCD method, we will seek $\mQ$ such that $\mA \mP_\mB$ is a good low-rank approximation of $\mA$ in trace-norm.
\end{remark}

\subsection{Convergence rate and block size}

In many cases, $\E[\mZ]$ is difficult to compute or even estimate. However, the expected projection matrix $\E[\mZ_1]$ corresponding to the sketching matrix $\mS_1$ with a single row is often exactly computable.

In this section, we consider the special class of sketching matrices where $\mS \in \reals^{\ell \times m}$ consists of $\ell$ independent and identically distributed rows $\vs_1, \ldots, \vs_{\ell} \in \reals^m$.
This encompasses a range of practical sketching schemes, including block versions of the randomized Kaczmarz~\cite{StrohmerVershynin2009, NeedellTropp2014} and coordinate descent~\cite{LeventhalLewis2010} methods. The following result shows that in this setting, the iteration complexity improves at least linearly in the block size $\ell$ compared to the single-row case; i.e., if using $\mS_1 = \vs_1^{\tran}$ requires $\tau$ iterations to reach a desired tolerance in mean squared error, then using $\mS$ requires at most $\tau / \ell$ iterations.

\begin{proposition} \label{prop:convergence_rate_block}
Suppose that $\mS_{\ell} \in \reals^{\ell \times m}$ is a matrix with i.i.d.\ random rows $\vs_1, \ldots, \vs_\ell \in \reals^{m}$, and let $\mZ_{1} = (\vs_1^{\tran} \mA \mB^{-1/2} \mP)^{\dagger} \vs_1^{\tran} \mA \mB^{-1/2} \mP$ be the orthogonal projector onto $\range(\mP \mB^{-1/2} \mA^{\tran} \mS_{1}^{\tran})$.
Assume that $\nullspace(\E[\mZ_1]) = \nullspace(\mA \mB^{-1/2} \mP)$.
If $\{ \vx^k \}_{k \geq 1}$ are the subspace-constrained sketch-and-project iterates with sketching matrix $\mS_{\ell}$, then
\begin{equation} \label{eq:convergence_rate_block:2}
    \E \norm{\vx^k - \vx^*}_{\mB}^2 \leq \bigl(1 - \lambda^+_{\mathrm{min}}(\E[\mZ_1]) \bigr)^{\ell k} \cdot \norm{\vx^0 - \vx^*}_{\mB}^2.
\end{equation}
\end{proposition}

The key technical ingredient behind Proposition~\ref{prop:convergence_rate_block} is the following rank-one update formula for orthogonal projection matrices from~\cite{DeLiLiMa2020}, which is derived using a rank-one update formula for the Moore--Penrose pseudoinverse (\cite[Theorem~1]{Meyer1973}).

\begin{lemma}[{\cite[Lemma~1]{DeLiLiMa2020}}] \label{lem:projector_rankone_update}
Let $\mX \in \reals^{t \times n}$ and $\mX_{-t} \in \reals^{(t-1) \times n}$ be the matrix $\mX$ without its last row $\vx_t \in \reals^n$. Suppose that $\mPi = \mX^{\dagger} \mX$ and $\mPi_{-t} = \mX_{-t}^{\dagger} \mX_{-t}$ are the orthogonal projectors onto the ranges of $\mX^{\tran}$ and $\mX_{-t}^{\tran}$ respectively.
If $\vx_t^{\tran} (\mI - \mPi_{-t}) \vx_t \ne \vzero$, then 
\[
    \mPi - \mPi_{-t} = \frac{(\mI - \mPi_{-t}) \vx_t \vx_t^{\tran} (\mI - \mPi_{-t})}{\vx_t^{\tran} (\mI - \mPi_{-t}) \vx_t}.
\]
Otherwise, if $\vx_t^{\tran} (\mI - \mPi_{-t}) \vx_t = 0$, then $\vx_t \in \range(\mX_{-t}^{\tran})$ and $\mPi = \mPi_{-t}$, so the decomposition above also holds provided that the right hand side is interpreted as zero.
\end{lemma}

We can now prove Proposition~\ref{prop:convergence_rate_block} using Lemma~\ref{lem:projector_rankone_update} to decompose the orthogonal projector $\mZ$ into a sum of rank-one projections onto a growing sequence of subspaces.

\begin{proof}[Proof of Proposition~\ref{prop:convergence_rate_block}]
It will be useful to define the sketching matrices for all intermediate block sizes. For $1 \leq t \leq \ell$, let $\mS_t \in \reals^{t \times m}$ be the matrix with rows $\vs_1, \ldots, \vs_t$, and $\mX_t := \mS_{t} \mA \mB^{-1/2} \mP \in \reals^{t \times n}$ be the matrix with rows $\vx_1, \ldots, \vx_t$, where $\vx_i := \vs_i^{\tran} \mA \mB^{-1/2} \mP$. Furthermore, let $\mZ_{t} = \mX_{t}^{\dagger} \mX_{t}$ be the orthogonal projector onto the range of $\mX_{t}^{\tran}$ (with $\mZ_0 := \mzero$).
In particular, we get the following representation
\begin{equation} \label{eq:convergence_rate_block_pf:sizeone_proj}
    \mX_1^{\dagger} \mX_1 = \frac{\mP \mB^{-1/2} \mA \vs_{1} \vs_{1}^{\tran} \mA \mB^{-1/2} \mP}{\vs_{1}^{\tran} \mA \mB^{-1/2} \mP \mB^{-1/2} \mA \vs_{1}}.
\end{equation}
Next, observe that the exactness condition $\nullspace(\E[\mZ_{t}]) = \nullspace(\mA \mB^{-1/2} \mP)$ holds for all $\mZ_t$. Indeed, $\mZ_{t} \succeq \mZ_{1}$ and by the exactness assumption on $\mZ_1$, \[
    \nullspace(\E[\mZ_{t}]) \subseteq \nullspace(\E[\mZ_{1}]) = \nullspace(\mA \mB^{-1/2} \mP) \subseteq \nullspace(\E[\mZ_{t}]).
\]
Hence, all the inclusions are equalities.

We claim that for all $1 \leq t \leq \ell$ and unit vectors $\vy \in \range(\mP \mB^{-1/2} \mA^{\tran})$,
\begin{equation} \label{eq:convergence_rate_block_pf:claim}
    \vy^{\tran} \E[\mZ_{t}] \vy \geq 1 - \bigl( 1 - \lambda_{\mathrm{min}}^+(\E[\mZ_1]) \bigr)^{t}.
\end{equation}
We will prove~\eqref{eq:convergence_rate_block_pf:claim} by induction on $t$. When $t = 1$, the base case $\vy^{\tran} \E[\mZ_1] \vy \geq \lambda_{\mathrm{min}}^+(\E[\mZ_1])$ follows because $\vy \in \range(\E[\mZ_1])$.
Assuming that~\eqref{eq:convergence_rate_block_pf:claim} holds for $t - 1 < \ell$, we want to prove that~\eqref{eq:convergence_rate_block_pf:claim} holds for $t$.
Using Lemma~\ref{lem:projector_rankone_update}, we have the following decomposition of $\mZ_{t}$:
\begin{equation} \label{eq:convergence_rate_block_pf:decomp}
    \mZ_{t}
    = \mZ_{t-1} + (\mZ_{t} - \mZ_{t-1})
    = \mZ_{t-1} + \frac{(\mI - \mZ_{t-1}) \mP \mB^{-1/2} \mA \vs_{t} \vs_{t}^{\tran} \mA \mB^{-1/2} \mP (\mI - \mZ_{t-1})}{\vs_{t}^{\tran} \mA \mB^{-1/2} \mP (\mI - \mZ_{t-1}) \mP \mB^{-1/2} \mA \vs_{t}}.
\end{equation}
Observe that $\vs_{t}^{\tran} \mA \mB^{-1/2} \mP (\mI - \mZ_{t-1}) \mP \mB^{-1/2} \mA \vs_{t} \leq \vs_{t}^{\tran} \mA \mB^{-1/2} \mP \mB^{-1/2} \mA \vs_{t}$. Hence, by taking expectation over the randomness in $\vs_{t}$ only, conditional on $\vs_1, \dots, \vs_{t-1}$, and using linearity of expectation, we obtain
\begin{align*}
    &\vy^{\tran} \E_{\vs_t \mid \vs_1, \ldots, \vs_{t-1}}[\mZ_{t}] \vy \\
    &\quad\geq \vy^{\tran} \mZ_{t-1} \vy + \vy^{\tran} (\mI - \mZ_{t-1}) \E_{\vs_{t} \mid \vs_1, \ldots, \vs_{t-1}} \left[ \frac{\mP \mB^{-1/2} \mA \vs_{t} \vs_{t}^{\tran} \mA \mB^{-1/2} \mP}{\vs_{t}^{\tran} \mA \mB^{-1/2} \mP \mB^{-1/2} \mA \vs_{t}} \right] (\mI - \mZ_{t-1}) \vy \\
    &\quad= \vy^{\tran} \mZ_{t-1} \vy + \vy^{\tran} (\mI - \mZ_{t-1}) \E[\mZ_1] (\mI - \mZ_{t-1}) \vy.
\end{align*}
The last equality follows from comparing the random matrix inside the inner expectation to~\eqref{eq:convergence_rate_block_pf:sizeone_proj} and using the fact that $\vs_{t}$ has the same distribution as $\vs_1$.
Therefore, since $(\mI - \mZ_{t-1}) \vy \in \range(\mP \mB^{-1/2} \mA^{\tran})$ by definition of $\mZ_{t}$, and $\mZ_{t-1}$ is an orthogonal projector, this implies that
\begin{align*}
    \vy^{\tran} \E_{\vs_t \mid \vs_1, \ldots, \vs_{t-1}}[\mZ_{t}] \vy
    \geq \vy^{\tran} \mZ_{t-1} \vy + \lambda_{\mathrm{min}}^+(\E[\mZ_1]) \cdot \vy^{\tran} (\mI - \mZ_{t-1}) \vy.
\end{align*}
By taking the full expectation and using the induction hypothesis, we deduce that
\begin{align*}
    \vy^{\tran} \E[\mZ_{t}] \vy
    &\geq \lambda_{\mathrm{min}}^+(\E[\mZ_1]) + \bigl( 1 - \lambda_{\mathrm{min}}^+(\E[\mZ_1]) \bigr) \cdot \vy^{\tran} \E[\mZ_{t-1}] \vy \\
    &\geq \lambda_{\mathrm{min}}^+(\E[\mZ_1]) + \bigl( 1 - \lambda_{\mathrm{min}}^+(\E[\mZ_1]) \bigr) \cdot \bigl( 1 - \bigl( 1 - \lambda_{\mathrm{min}}^+(\E[\mZ_1]) \bigr)^{t-1} \bigr) \\
    &= 1 - \bigl( 1 - \lambda_{\mathrm{min}}^+(\E[\mZ_1]) \bigr)^{t},
\end{align*}
and so, $1 - \lambda_{\mathrm{min}}^+(\E[\mZ_{\ell}]) \leq \bigl( 1 - \lambda_{\mathrm{min}}^+(\E[\mZ_1]) \bigr)^{\ell}$.
Employing Theorem~\ref{thm:convergence_rate} completes the proof.
\end{proof}

\paragraph{Example: Randomized block Kaczmarz.}
The subspace-constrained sketch-and-project method generalizes the \emph{subspace-constrained randomized Kaczmarz} (SCRK) method analyzed in~\cite{LokRebrova2024}, which can be described as follows.
Let $\mB = \mI$ and $\mQ = \ve_{\mathcal{S}}^{\tran}$ for some $\mathcal{S} \subseteq [m]$ so that the iterates are confined within the solution space of the subsystem $\mA_{\mathcal{S},:} \vx = \vb_{\mathcal{S}}$, and $\mP = \mI - (\mA_{\mathcal{S},:})^{\dagger} \mA_{\mathcal{S},:}$ be the projector onto $\nullspace(\mA_{\mathcal{S},:})$.
From Lemma~\ref{lem:update}, the updates of the SCRK method are given by
\[
    \vx^{k+1} = \vx^k - \frac{\va_j^{\tran} \vx^k - \vb_j}{\norm{\mP \va_j}_2^2} \mP \va_j.
\]
Suppose that each row $\va_j^{\tran} = \mA_{j,:}$ of the matrix $\mA$ is independently sampled with probability $\norm{\mP \va_j}_2^2 / \norm{\mA \mP}_F^2$ in each iteration. 
Since $\mZ = \mP \va_j \va_j^{\tran} \mP / \norm{\mP \va_j}_2^2$ if row $j$ is sampled, it follows that $\E[\mZ] = \mP \mA^{\tran} \mA \mP / \norm{\mA \mP}_F^2$. Note that the exactness condition~\eqref{eq:exactness} holds trivially because $\nullspace(\mP \mA^{\tran} \mA \mP) = \nullspace(\mA \mP)$.

More generally, if a block $\mathcal{J} \subseteq [m]$ of $\ell \geq 1$ rows are sampled independently as above, then the block update
\begin{equation} \label{eq:sc_block_rk_update}
    \vx^{k+1} = \vx^k - (\mA_{\mathcal{J},:} \mP)^{\dagger} (\mA_{\mathcal{J},:} \vx^k - \vb_{\mathcal{J}}).
\end{equation}
describes a subspace-constrained \emph{block} randomized Kaczmarz method. From Proposition~\ref{prop:convergence_rate_block}, the following convergence rate bound holds:

\begin{corollary} \label{cor:sc_block_rk_rate}
Let $\{ \vx^k \}_{k \geq 0}$ be the iterates defined by~\eqref{eq:sc_block_rk_update}, where $\vx^0$ solves $\mA_{\mathcal{S},:} \vx = \vb_{\mathcal{S}}$ and the block $\mathcal{J} = \{ j_1, \ldots, j_\ell \}$ in each iteration consists of $\ell$ rows independently sampled with the probabilities $\{ \norm{\mP \va_j}_2^2 / \norm{\mA \mP}_F^2 \}_{j=1}^m$. Then
\begin{equation} \label{eq:sc_block_rk_rate}
    \E \norm{\vx^k - \vx^*}_2^2 \leq \left( 1 - \frac{\sigma_{\mathrm{min}}^+(\mA \mP)^2}{\norm{\mA \mP}_F^2} \right)^{\ell k} \cdot \norm{\vx^0 - \vx^*}_2^2,
\end{equation}
where $\sigma^+_{\mathrm{min}}(\mA \mP)^2 = \lambda^+_{\mathrm{min}}(\mP \mA^{\tran} \mA \mP)$ is the smallest non-zero squared singular value of $\mA \mP$.
\end{corollary}

By taking $\mP = \mI$, Corollary~\ref{cor:sc_block_rk_rate} implies Proposition~\ref{prop:block_rk_rate} for the convergence rate of the randomized block Kaczmarz method.

\section{Analysis of subspace-constrained randomized coordinate descent} \label{sec:sc-rcd}

In this section, we analyze the subspace-constrained randomized coordinate descent (SC-RCD) method for solving the linear system $\mA \vx = \vb$, where $\mA \in \reals^{n \times n}$ is a positive semidefinite matrix.

\subsection{Convergence rate of SC-RCD}

\paragraph{Update formulas.}
In the case when $\mA$ is also positive definite, the SC-RCD method is an instance of the subspace-constrained sketch-and-project method with geometry parameter $\mB = \mA$ and sparse sketching matrices aligned with the standard coordinate basis of the form $\mS = \ve_{\mathcal{J}}^{\tran}$, defined as in~\eqref{eq:block_rcd_update}, for randomly chosen subsets $\mathcal{J} \subseteq [n]$. Furthermore, the subspace constraint is defined by the matrix $\mQ = \ve_{\mathcal{S}}^{\tran} \in \reals^{d \times n}$, parameterized by a (small) subset $\mathcal{S} \subseteq [n]$ of $d \ll n$ coordinates representing $d$ salient data points or landmarks, which we propose to efficiently choose using an algorithm such as RPCholesky.

Recall that $\vx^0 \in \reals^n$ is any initial iterate satisfying $\mA_{\mathcal{S},:} \vx^0 = \vb_{\mathcal{S}}$.
By Lemma~\ref{lem:update}, the update formula of the SC-RCD method is given by
\begin{equation}\label{eq:rcd-update-1}
    \vx^{k+1} = \vx^k - \mA^{-1/2} \mP \mA^{1/2} \ve_J (\ve_J^{\tran} \mA^{1/2} \mP \mA^{1/2} \ve_J)^{\dagger} (\mA_{J,:} \vx^k - \vb_{J}),
\end{equation}
where $\mP = \mI - \mA^{1/2} \ve_{\mathcal{S}} (\ve_{\mathcal{S}}^{\tran} \mA \ve_{\mathcal{S}})^{\dagger} \ve_{\mathcal{S}}^{\tran} \mA^{1/2}$ is the orthogonal projector onto $\nullspace(\ve_{\mathcal{S}}^{\tran} \mA^{1/2})$.

Note that the rank-$d$ Nystr\"{o}m approximation of $\mA$ with respect to the coordinates indexed by $\mathcal{S}$ (e.g., see~\cite[\S 19.2]{MartinssonTropp2020}, \cite[\S 2.1]{ChenEtAl2024}) can be written as
\begin{equation}
    \mA \langle \mathcal{S} \rangle = \mA_{:,\mathcal{S}} (\mA_{\mathcal{S},\mathcal{S}})^{\dagger} \mA_{\mathcal{S},:}.
\end{equation}
Using this expression, we can make the key observation that the residual matrix satisfies
\begin{equation} \label{eq:nystrom_proj}
    \mA^{\circ} = \mA - \mA \langle \mathcal{S} \rangle = \mA^{1/2} \mP \mA^{1/2}.
\end{equation}
Thus, the update formula~\eqref{eq:rcd-update-1} can be written as
\begin{align} \label{eq:sc-rcd_update}
    \vx^{k+1}
    &= \vx^k - (\ve_{\mathcal{J}} - \ve_{\mathcal{S}} (\mA_{\mathcal{S},\mathcal{S}})^{\dagger} \mA_{\mathcal{S},:} \ve_{\mathcal{J}}) (\mA^{\circ}_{\mathcal{J},\mathcal{J}})^{\dagger} (\mA_{\mathcal{J},:} \vx^k - \vb_{\mathcal{J}}) \nonumber\\
    &= \vx^k - \ve_{\mathcal{J}} \valpha^k + \ve_{\mathcal{S}} \vbeta^k,
\end{align}
where, with the corresponding residual vector denoted by $\vr^k = \mA \vx^k - \vb$,
\[
    \valpha^k := (\mA^{\circ}_{\mathcal{J},\mathcal{J}})^{\dagger} \vr^k_{\mathcal{J}},
    \quad
    \vbeta^k := \mC_{:,\mathcal{J}} \valpha^k,
    \quad\text{and}\quad
    \mC := (\mA_{\mathcal{S},\mathcal{S}})^{\dagger} \mA_{\mathcal{S},:}.
\]
Note that the formula~\eqref{eq:sc-rcd_update} remains well-defined even if $\mA$ is rank deficient.
Hence, the SC-RCD method can be defined directly through the update formula~\eqref{eq:sc-rcd_update}, and remains well-defined in the general positive semidefinite case.

With the update~\eqref{eq:sc-rcd_update} for the iterate $\vx^k$, the following update for the residual vector $\vr^{k+1} = \mA \vx^{k+1} - \vb$ can be derived:
\begin{align}
    \vr^{k+1} 
    &= (\mA \vx^k - \vb) - \mA^{\circ}_{:,\mathcal{J}} (\mA^{\circ}_{\mathcal{J},\mathcal{J}})^{\dagger} (\mA_{\mathcal{J},:} \vx^k - \vb_{\mathcal{J}}) \nonumber\\
    &= \vr^{k} - \mA^{\circ}_{:,\mathcal{J}} (\mA^{\circ}_{\mathcal{J},\mathcal{J}})^{\dagger} \vr^k_{\mathcal{J}} \nonumber\\
    &= \vr^{k} - \mA^{\circ}_{:,\mathcal{J}} \valpha^k. \label{eq:sc-rcd_update_resid}
\end{align}
Finally, using~\eqref{eq:nystrom_proj}, the orthogonal projector $\mZ$ onto $\range(\mP \mA^{1/2} \ve_{\mathcal{J}})$ can be written
\begin{align}
    \mZ
    &= \mP \mA^{1/2} \ve_{\mathcal{J}} (\ve_{\mathcal{J}}^{\tran} \mA^{1/2} \mP \mA^{1/2} \ve_{\mathcal{J}})^{\dagger} \ve_{\mathcal{J}}^{\tran} \mA^{1/2} \mP \nonumber\\
    &= \mP \mA^{1/2} \ve_{\mathcal{J}} (\mA^{\circ}_{\mathcal{J},\mathcal{J}})^{\dagger} \ve_{\mathcal{J}}^{\tran} \mA^{1/2} \mP. \label{eq:sc-rcd_Z}
\end{align}

\begin{remark}[Properties of the Nystr\"{o}m approximation] \label{rmk:nystrom_properties} 
The residual matrix $\mA^{\circ} = \mA - \mA \langle \mathcal{S} \rangle$ is the (generalized) \emph{Schur complement} of $\mA$ with respect to the coordinates $\mathcal{S} \subseteq [n]$ (e.g., see~\cite{HornZhang2005schur}). From this observation, the following well-known properties of the Nystr\"{o}m approximation $\mA \langle \mathcal{S} \rangle$ can be derived.
(i) $\mA^{\circ} \succeq \mzero$, so $\mA \succeq \mA \langle \mathcal{S} \rangle \succeq \mzero$.
(ii) The columns indexed by $\mathcal{S}$ in $\mA \langle \mathcal{S} \rangle$ and $\mA$ are equal: $(\mA \langle \mathcal{S} \rangle)_{:,\mathcal{S}} = \mA_{:,\mathcal{S}}$.
(iii) The range of $\mA \langle \mathcal{S} \rangle$ coincides with the span of the columns $\mathcal{S}$ in $\mA$: $\range(\mA \langle \mathcal{S} \rangle) = \range(\mA_{:,\mathcal{S}})$.
(iv) If the $d \times d$ block $(\mA \langle \mathcal{S} \rangle)_{\mathcal{S},\mathcal{S}}$ is invertible, then $\mathrm{rank}(\mA^{\circ}) = \mathrm{rank}(\mA) - d$ (\cite[Theorem~1.6]{HornZhang2005schur}), and the eigenvalues of $\mA^{\circ}$ interlace those of $\mA$: $\lambda_i(\mA) \geq \lambda_i(\mA^{\circ}) \geq \lambda_{i+d}(\mA)$ for $1 \leq i \leq n - d$ (\cite[Theorem~2.1]{Liu2005schur}).\footnote{If $(\mA \langle \mathcal{S} \rangle)_{\mathcal{S},\mathcal{S}}$ is not invertible, these properties still hold and can be proved using the same arguments using the generalized Aitken block-diagonalization formula (\cite[Eq.~(6.0.20)]{PuntanenStyan2005schur}) instead.}
(v) Finally, if $\mA \langle \mathcal{S} \rangle$ is computed using RPCholesky, then $(\mA \langle \mathcal{S} \rangle)_{\mathcal{S},\mathcal{S}}$ is invertible by virtue of the adaptive diagonal sampling process for the pivots.
\end{remark}

\paragraph{Convergence rate.}
In the positive definite case, if the block $\mathcal{J} \subseteq [n]$ is formed by the simple scheme of sampling $\ell$ coordinates independently from the same distribution, the convergence rate of the SC-RCD method can be deduced from the general theory for the subspace-constrained sketch-and-project method (Theorem~\ref{thm:convergence_rate} and Proposition~\ref{prop:convergence_rate_block}).
The following result obtains the resulting rate when the coordinates are sampled proportionally to the diagonal of the residual matrix $\mA^{\circ}$:

\begin{theorem}[Diagonal sampling] \label{thm:sc-rcd_convrate}
Let $\mA \in \reals^{n \times n}$ be a positive semidefinite matrix and $\vx^*$ be any solution of $\mA \vx = \vb$.
Suppose that $\{ \vx^k \}_{k \geq 0}$ are the iterates defined by~\eqref{eq:sc-rcd_update} with a fixed subset $\mathcal{S} \subseteq [n]$, and the block $\mathcal{J} = \{ j_1, \ldots, j_\ell \}$ in each iteration consists of $\ell$ coordinates independently sampled according to the distribution $\mathrm{diag}(\mA^{\circ}) / \tr(\mA^{\circ})$. Then
\[
    \E \norm{\vx^k - \vx^*}_{\mA}^2 \leq \left( 1 - \frac{\lambda_{\mathrm{min}}^+(\mA^{\circ})}{\tr(\mA^{\circ})} \right)^{k \ell} \cdot \norm{\vx^0 - \vx^*}_{\mA}^2.
\]
\end{theorem}

\begin{proof}
First, suppose that $\mA$ is positive definite. Since the blocks consist of i.i.d.\ samples, the SC-RCD method fits into the framework of Proposition~\ref{prop:convergence_rate_block}. Hence, it suffices to analyze $\E[\mZ_1]$, where $\mZ_1$ is the orthogonal projector onto $\range(\mP \mA^{1/2} \ve_j)$ from~\eqref{eq:sc-rcd_Z} with $j \in [n]$ and block size $\ell = 1$.

The key idea is that the expectation $\E[\mZ_1]$ admits a nice formula in terms of the residual matrix when the coordinate $j \in [n]$ is sampled with probability $\mA^{\circ}_{j,j} / \tr(\mA^{\circ})$ in each iteration:
\begin{align}
    \E[\mZ_1]
    &= \sum_{j: \mA^{\circ}_{j,j} > 0} \frac{\mA^{\circ}_{j,j}}{\tr(A^{\circ})} \cdot \frac{1}{\mA^{\circ}_{j,j}} \mP \mA^{1/2} \ve_j \ve_j^{\tran} \mA^{1/2} \mP \nonumber \\
    &= \frac{1}{\tr(\mA^{\circ})} \mP \mA^{1/2} \left( \mI - \sum_{j: \mA^{\circ}_{j,j} = 0} \ve_j \ve_j^{\tran} \right) \mA^{1/2} \mP
    = \frac{1}{\tr(\mA^{\circ})} \mP \mA \mP. \label{eq:sc-rcd_sizeone_convrate:1}
\end{align}
Here we use the fact that $\sum_{j=1}^n \ve_j \ve_j^{\tran} = \mI$, and $\mA^{\circ}_{j,j} = 0$ implies $\mP \mA^{1/2} \ve_j \ve_j^{\tran} \mA^{1/2} \mP = \mzero$, since this is a rank one, psd matrix with zero trace.
Note that the exactness condition is trivially satisfied because
\[
    \nullspace(\E[\mZ_1]) = \nullspace((\mA^{1/2} \mP)^{\tran} \mA^{1/2} \mP) = \nullspace(\mA^{1/2} \mP).
\]
Thus, by Proposition~\ref{prop:convergence_rate_block}, SC-RCD with block size $\ell$ results in the expected error
\[
    \E\norm{\vx^k - \vx^*}_{\mA}^2 \leq \left( 1 - \frac{\lambda_{\mathrm{min}}^+(\mP \mA \mP)}{\tr(\mA^{\circ})} \right)^{k \ell} \cdot \norm{\vx^0 - \vx^*}_{\mA}^2.
\]
To conclude, we observe that the eigenvalues of $\mP \mA \mP$ and $\mA^{1/2} \mP \mA^{1/2} = \mA^{\circ}$ 
coincide.

Finally, we claim that the same result holds if $\mA$ is positive semidefinite but not invertible. In this case, the solution $\vx^*$ is not unique since it can be shifted by any vector in $\nullspace(\mA)$. However, we only need to track $\norm{\vx^k - \vx^*}_{\mA}^2$, which is constant for any choice of solution $\vx^*$.\footnote{For example, we can write $\norm{\vx^k - \vx^*}_{\mA}^2 = f(\vx^k) - f(\vx^*)$, where $f: \vx \mapsto \vx^{\tran} \mA \vx - 2 \vb^{\tran} \vx$ (\cite[Eq.~(8)]{LeventhalLewis2010}). Since $\nabla f(\vx) = 2(\mA \vx - \vb)$, any solution $\vx^*$ of $\mA \vx = \vb$ with psd $\mA$ is a minimizer of the convex quadratic $f$.}
Although we cannot directly use the results from the subspace-constrained sketch-and-project framework, it can be shown by direct calculation that if we \emph{define} the orthogonal projector $\mZ$ by~\eqref{eq:sc-rcd_Z}, then the fixed point equation $\mA^{1/2}(\vx^{k+1} - \vx^*) = (\mI - \mZ) \mA^{1/2}(\vx^k - \vx^*)$ still holds. Hence, the same arguments in Theorem~\ref{thm:convergence_rate} and Proposition~\ref{prop:convergence_rate_block} can be used to reach the same conclusions.
\end{proof}

If the blocks are sampled uniformly at random instead, then we can obtain a similar result, except that the rate depends on the spectrum of the diagonal-normalized residual matrix. The proof is similar and will be omitted.

\begin{proposition}[Uniform sampling] \label{prop:sc-rcd_convrate_unif}
Consider the same setup as in Theorem~\ref{thm:sc-rcd_convrate}. If instead, the block $\mathcal{J}$ consists of $\ell$ coordinates independently sampled from $\{ j \in [n]: \mA^{\circ}_{j,j} > 0 \}$ uniformly at random in each iteration, then with $\mD \in \reals^{n \times n}$ the diagonal matrix with entries $\mD_{j,j} = \mA^{\circ}_{j,j}$,
\[
    \E \norm{\vx^k - \vx^*}_{\mA}^2 \leq \left( 1 - \frac{\lambda_{\mathrm{min}}^+(\mD^{\dagger/2} \mA^{\circ} \mD^{\dagger/2})}{n - d} \right)^{k \ell} \cdot \norm{\vx^0 - \vx^*}_{\mA}^2.
\]
\end{proposition}

Note that when $\ell = 1$, the convergence rate in Theorem~\ref{thm:sc-rcd_convrate} is analogous to the bound~\eqref{eq:rcd_conv_bdd} proved for randomized coordinate descent, which depends on $\lambda^+_{\mathrm{min}}(\mA) / \tr(\mA)$ and is tight in general.
By using the fact that the eigenvalues of the residual matrix $\mA^{\circ}$ interlace those of $\mA$ (Remark~\ref{rmk:nystrom_properties}), or directly applying the variational principle, the numerator of the rate is guaranteed to be no smaller: $\lambda^+_{\mathrm{min}}(\mA^{\circ}) \geq \lambda^+_{\mathrm{min}}(\mA)$.
However, a significant improvement can be realized if the denominator is much smaller; i.e., $\tr(\mA^{\circ}) \ll \tr(\mA)$. This depends on the quality of the low-rank approximation in trace-norm and relates to how the pivots $\mathcal{S}$ are selected.
When $\mathcal{S}$ is selected using RPCholesky, we are able to use its approximation guarantees to prove Theorem~\ref{thm:sc-rcd_conv_main}.

\begin{proof}[Proof of Theorem~\ref{thm:sc-rcd_conv_main}]
Let $\mA \langle \mathcal{S} \rangle$ be the Nystr\"{o}m approximation of $\mA$ output by RPCholesky with randomly sampled pivot set $\mathcal{S}$. For a given $\mathcal{S}$, Theorem~\ref{thm:sc-rcd_convrate} implies that the error of the SC-RCD method satisfies
\[
    \E\left[ \norm{\vx^k - \vx^*}_{\mA}^2 \mid \mathcal{S} \right]
    \leq \left( 1 - \frac{\lambda_{\mathrm{min}}^+(\mA - \mA \langle \mathcal{S} \rangle)}{\tr(\mA - \mA \langle \mathcal{S} \rangle)} \right)^{k \ell} \cdot \norm{\vx^0 - \vx^*}_{\mA}^2.
\]
Note that $\lambda_{\mathrm{min}}^+(\mA - \mA \langle S \rangle) \geq \lambda_{\mathrm{min}}^+(\mA)$. By using Theorem~\ref{thm:rpcholesky_error} and Markov's inequality, the event $\mathcal{E}$ where $\tr(\mA - \mA \langle \mathcal{S} \rangle) \leq \rho^{-1} (1 + \delta) \sum_{i > r} \lambda_i(A)$ occurs with probability at least $1 - \rho$. Conditional on this event, substituting this bound into the displayed equation above completes the proof.
\end{proof}

\subsection{Implementation and complexity of SC-RCD} \label{sec:sc-rcd_implementation}

A practical implementation of the updates~\eqref{eq:sc-rcd_update} and~\eqref{eq:sc-rcd_update_resid} of the SC-RCD method is summarized in the pseudocode in Algorithm~\ref{alg:sc-rcd}, presented in Section~\ref{sec:main_results}.
Its efficiency, as well as its complexity estimates used in Theorem~\ref{thm:sc-rcd_flat}, relies on several computational considerations discussed below. 

\begin{enumerate}[label=(\roman*), ref=\roman*, leftmargin=*]
    \item \label{rmk:sc-rcd_implementation_initial}
    The partial pivoted Cholesky factor $\mF \in \reals^{n \times d}$ output by RPCholesky can be used to efficiently compute an initial iterate $\vx^0$ solving $\mA_{\mathcal{S},:} \vx^0 = \vb_{\mathcal{S}}$ (line~\ref{alg:sc-rcd_initial_solve}).
    From Remark~\ref{rmk:nystrom_properties}, observe that $\mA_{\mathcal{S},\mathcal{S}} = (\mA \langle \mathcal{S} \rangle)_{\mathcal{S},\mathcal{S}} = \mF_{\mathcal{S},:} (\mF_{\mathcal{S},:})^{\tran}$, where $\mF_{\mathcal{S},:} \in \reals^{d \times d}$ is an invertible, lower triangular matrix, and so $(\mA_{\mathcal{S},\mathcal{S}})^{-1} = (\mF_{\mathcal{S},:})^{-\tran} (\mF_{\mathcal{S},:})^{-1}$.
    Hence, given any vector $\vx \in \reals^n$, we can compute
    \[
        \vx^0 = \vx - \ve_{\mathcal{S}} (\mA_{\mathcal{S},\mathcal{S}})^{-1} (\mA_{\mathcal{S},:} \vx - \vb_{\mathcal{S}})
    \]
    by only modifying the coordinates of $\vx$ in $\mathcal{S}$: first, we solve $\mF_{\mathcal{S},:} \vw = \mA_{\mathcal{S},:} \vx - \vb_{\mathcal{S}}$ for $\vw \in \reals^d$ using forward substitution, and then solve $(\mF_{\mathcal{S},:})^{\tran} \vbeta = \vw$ for $\vbeta \in \reals^d$ using back substitution. Finally, we set $\vx^0 \leftarrow \vx$ and $\vx^0_{\mathcal{S}} \leftarrow \vx^0_{\mathcal{S}} - \vbeta$.
    In total, this requires $O(d^2)$ flops, which represents a substantial improvement over the $O(d^2 n)$ flops from solving $\mA_{\mathcal{S},:} \vx = \vb_{\mathcal{S}}$ naively.

    Given any vector $\vx$ and $\vr = \mA \vx - \vb$ (e.g., $\vx = \vzero$ and $\vr = -\vb$), the residual vector $\vr^0$ associated with $\vx^0$ can also be computed by $\vr^0 \leftarrow \vr - \mA_{:,\mathcal{S}} \vbeta$, which costs $O(d n)$ flops, instead of $O(n^2)$ flops from computing $\vr^0 = \mA \vx^0 - \vb$ directly.

    \item \label{rmk:sc-rcd_implementation_compute_B}
    Similarly, the columns of the auxiliary matrix $\mC = (\mA_{\mathcal{S},\mathcal{S}})^{\dagger} \mA_{\mathcal{S},:} = (\mF_{\mathcal{S}, :})^{-\tran} \mF^{\tran} \in \reals^{d \times n}$ (line~\ref{alg:sc-rcd_compute_B}) can be computed by solving a sequence of upper triangular linear systems with back substitution (each requiring $O(d^2)$ flops). Specifically, for each $j \in [n] \setminus \mathcal{S}$, the $j$th column $\mC_{:,j}$ is the solution of $(\mF_{\mathcal{S},:})^{\tran} \mC_{:,j} = (\mF_{j,:})^{\tran}$. (The submatrix $\mC_{:,\mathcal{S}}$ is the identity, but it is not used.) In total, computing $\mC$ with this procedure requires $O(d^2 (n - d))$ flops.

    \item \label{rmk:sc-rcd_implementation_sampling_withoutrep}
    The block $\mathcal{J} \subseteq [n]$ of $\ell$ coordinates can also be sampled \emph{without replacement} (line~\ref{alg:sc-rcd_sampling_step}) with probabilities proportional to $\mathrm{diag}(\mA^{\circ})$, which will always result in a larger error decrease in each iteration (see the upcoming Remark~\ref{rmk:sampling_without_replacement}).
    Another even simpler alternative is to sample the block $\mathcal{J}$ of $\ell$ coordinates \emph{uniformly at random} (with or without replacement), which will be especially effective if $\mA^{\circ}$ has incoherence properties (e.g., see~\cite[Lemma 10]{DerezinskiEtAl2024fine}).

    \item \label{rmk:sc-rcd_implementation_inexact}
    Line~\ref{alg:sc-rcd_projection_step} computes $\valpha^{k-1} = (\mA^{\circ}_{\mathcal{J}, \mathcal{J}})^{\dagger} \vr^{k-1}_{\mathcal{J}}$ by finding the min-norm solution of the $\ell \times \ell$ linear system $\mA^{\circ}_{\mathcal{J}, \mathcal{J}} \valpha = \vr^{k-1}_{\mathcal{J}}$.\footnote{The linear system has a solution since $\vr^k_{\mathcal{J}} \in \range(\mA^{\circ}_{\mathcal{J}, \mathcal{J}})$. This follows from using the fact that $\mP \mA^{1/2}(\vx^k - \vx^*) = \mA^{1/2}(\vx^k - \vx^*)$ to write $\vr^k_{\mathcal{J}} = \ve_{\mathcal{J}}^{\tran} \mA^{1/2} \mA^{1/2} (\vx^k - \vx^*) = \ve_{\mathcal{J}}^{\tran} \mA^{1/2} \mP \mA^{1/2} (\vx^k - \vx^*) = \mA^{\circ}_{\mathcal{J},:} (\vx^k - \vx^*)$. Finally, we conclude by noting that $\range(\mA^{\circ}_{\mathcal{J},:}) \subseteq \range(\mA^{\circ}_{\mathcal{J}, \mathcal{J}})$ since $\mA^{\circ}$ is psd (\cite[Theorem~1.20]{HornZhang2005schur}).}
    For small block sizes $\ell$, $\valpha$ can be solved directly using a method based on QR or SVD using $O(\ell^3)$ flops.
    For larger block sizes, $\valpha$ can be computed inexactly using an iterative method such as CG, noting that $\mA^{\circ}_{\mathcal{J}, \mathcal{J}}$ is psd (see the upcoming Remark~\ref{rmk:inexact_proj}).
\end{enumerate}

\paragraph{Complexity estimates.}
Based on the considerations above, the computational costs of SC-RCD (Algorithm~\ref{alg:sc-rcd}) can be analyzed in two stages:
\begin{itemize}[leftmargin=2em]
    \item \textbf{Initialization:} learning the rank-$d$ Nystr\"{o}m approximation $\mA \langle \mathcal{S} \rangle = \mF \mF^{\tran}$ using RPCholesky requires $O(d^2 n)$ arithmetic operations, $O(dn)$ entry evaluations of $\mA$, and $O(dn)$ storage~\cite{ChenEtAl2024, EpperlyEtAl2024rejection}. 
    Note that the normalized diagonal of the residual matrix $\mA^{\circ}$ can be read off the output of RPCholesky to obtain the sampling probabilities $\vp$ without any additional cost.
    Next, the initial iterate $\vx$ and auxiliary matrix $\mC$ can be computed with $O(d^2)$ and $O(d^2 n)$ operations respectively, and $dn$ entries of $\mC$ have to be stored in memory.
    
    \item \textbf{Iterations:} in each iteration, the residual submatrix $\mA^{\circ}_{\mathcal{J}, \mathcal{J}} = \mA_{\mathcal{J}, \mathcal{J}} - \mF_{\mathcal{J},:} (\mF_{\mathcal{J},:})^{\tran}$ is needed for the projection step. If $\mA^{\circ}$ cannot be stored in memory, this requires accessing and storing $\ell n$ entries of $\mA$ in $\mA_{:, \mathcal{J}}$ (since the entire columns are also needed to update the residual vector $\vr$) and $O(\ell^2 d)$ arithmetic operations.
    Then, solving for $\valpha$ requires $O(\ell^3)$ operations (possibly fewer if solved inexactly), computing $\vbeta$ requires $O(\ell d)$ operations, and updating the iterate $\vx$ and residual vector $\vr$ requires $\ell + d$ and $O(\ell n) + O(\ell d + dn)$ operations respectively.
    In summary, each iteration requires $O(dn + \ell n + \ell^3 + \ell^2 d)$ operations. If $\ell = O(\sqrt{n})$, then this simplifies to $O((\ell + d) n)$ operations per iteration. 
\end{itemize}

\medskip

Combining Theorem~\ref{thm:sc-rcd_conv_main} with the analysis of the computational costs of the SC-RCD method allows us to prove Theorem~\ref{thm:sc-rcd_flat} on the overall complexity of SC-RCD.

\begin{proof}[Proof of Theorem~\ref{thm:sc-rcd_flat}]
If RPCholesky is independently run $T = \lceil \log_2(2/\epsilon) \rceil$ times and the output $\mathcal{S}$ with the smallest residual trace-norm is chosen as in Remark~\ref{rmk:boosting_prob}, then the event $\mathcal{E}$ where $\tr(\mA - \mA \langle \mathcal{S} \rangle) \leq 2(1 + 1) \sum_{i > r} \lambda_i(\mA)$ occurs with probability at least $1 - 2^{-T} \geq 1 - \epsilon / 2$.
Applying Theorem~\ref{thm:sc-rcd_conv_main} with $\delta = 1$ implies that conditional on the event $\mathcal{E}$, the expected relative error after $k = \lceil 4 (n / \ell) \bar{\kappa}_r(\mA) \log(2/\epsilon) \rceil$ iterations satisfies
\begin{align*}
    \E\left[ \norm{\vx^k - \vx^*}_{\mA}^2 \mid \mathcal{E} \right]
    &\leq \exp \left( -4n \bar{\kappa}_r(\mA) \log(2/\epsilon) \cdot \frac{\lambda^+_{\mathrm{min}}(\mA)}{2(1 + 1) \sum_{i > r} \lambda_i(\mA)} \right) \cdot \norm{\vx^0 - \vx^*}_{\mA}^2 \\
    &\leq (\epsilon / 2) \cdot \norm{\vx^0 - \vx^*}_{\mA}^2,
\end{align*}
where we used the elementary inequality $1 - t \leq e^{-t}$ for the first inequality. By using the monotonicity property $\norm{\vx^k - \vx^*}_{\mA}^2 \leq \norm{\vx^0 - \vx^*}_{\mA}^2$ in the event that $\mathcal{E}$ does not hold, denoted by $\mathcal{E}^c$, the overall expectation can be bounded by
\begin{align*}
    \E \norm{\vx^k - \vx^*}_{\mA}^2
    &= \E\left[ \norm{\vx^k - \vx^*}_{\mA}^2 \mid \mathcal{E} \right] \cdot \prob{\mathcal{E}}
    + \E\left[ \norm{\vx^k - \vx^*}_{\mA}^2 \mid \mathcal{E}^c \right] \cdot \prob{\mathcal{E}^c} \\
    &\leq (\epsilon / 2) \cdot \norm{\vx^0 - \vx^*}_{\mA}^2 + (\epsilon / 2) \cdot \norm{\vx^0 - \vx^*}_{\mA}^2
    = \epsilon \cdot \norm{\vx^0 - \vx^*}_{\mA}^2,
\end{align*}
as desired.
It remains to compute the computational costs. Running RPCholesky $T = O(\log(1/\epsilon))$ times and initializing requires $O(d^2 n \log(1/\epsilon))$ arithmetic operations, $O(dn \log(1/\epsilon))$ entry evaluations, and $O(dn)$ storage.
Subsequently, each SC-RCD iteration requires accessing $\ell n$ entries of $\mA$ and $O(dn + \ell n + \ell^3 + \ell^2 d)$ arithmetic operations, so in total $O(n^2 \cdot \bar{\kappa}_r(\mA) \log(1/\epsilon))$ entry evaluations and $O((n^2 (d + \ell) / \ell + \ell^2 n + \ell d n) \cdot \bar{\kappa}_r(\mA) \log(1/\epsilon))$ arithmetic operations are required.
\end{proof}

\begin{remark}[Inexact projections] \label{rmk:inexact_proj} 
Note that each update~\eqref{eq:lem_update_1} requires finding the min-norm solution $\vz$ of the linear system $\mS \mA \mB^{-1/2} \mP \vz = \mS(\mA \vx^k - \vb)$ to compute $(\mS \mA \mB^{-1/2} \mP \mB^{-1/2} \mA^{\tran} \mS^{\tran})^{\dagger} \mS(\mA \vx^k - \vb)$. For practical efficiency, it is possible for an approximate solution to be computed; e.g., using an inner iterative method such as preconditioned CG.
See \cite[\S4.3]{DerezinskiYang2024}, \cite[\S6]{DerezinskiEtAl2024fine}, or~\cite{tappenden2016inexact} for results along these lines, where similar theoretical bounds as in Theorem~\ref{thm:convergence_rate} can be derived with the loss of a small multiplicative factor in the rate.
\end{remark}

\begin{remark}[Sampling without replacement] \label{rmk:sampling_without_replacement}
If the coordinates in the blocks are sampled without replacement in Theorem~\ref{thm:sc-rcd_convrate} or Proposition~\ref{prop:sc-rcd_convrate_unif}, then the same bounds hold because the convergence rate can only improve.
To see this, suppose that $\mathcal{J}$ and $\mathcal{J}'$ consists of $\ell$ coordinates sampled with and without replacement respectively, and let $\mZ \equiv \mZ(\mathcal{J})$ and $\mZ' \equiv \mZ'(\mathcal{J}')$ denote the corresponding orthogonal projectors onto $\range(\mP \mA^{1/2} \ve_{\mathcal{J}})$ and $\range(\mP \mA^{1/2} \ve_{\mathcal{J}'})$ from~\eqref{eq:sc-rcd_Z}.
The key observation is that $\mathcal{J}$ and $\mathcal{J}'$ can be coupled such that $\mathcal{J} \subseteq \mathcal{J}'$ by rejection sampling (e.g., $\mathcal{J}'$ can be formed by proposing the same indices sampled for $\mathcal{J}$ and resampling any duplicates), and hence $\mZ' \succeq \mZ$. Combined with Lemma~\ref{lem:error_decrease}, this implies that the error decrease in each iteration with $\mZ'$ is always at least as large as with $\mZ$.
\end{remark}

\section{Numerical experiments} \label{sec:numerical_results}

In this section, we present some numerical experiments demonstrating various features of the SC-RCD method. The experiments were performed using Python 3.12.7 on a 2.6 GHz Intel Skylake CPU with 32GB RAM. The code is available at \url{https://github.com/jackielok/subspace-constrained-rcd}.

\paragraph{Synthetic psd system.}

\begin{figure}[!htb]
    \centering
    \includegraphics[width=0.495\linewidth, trim={0.2cm 0.4cm 0.2cm 0.3cm}, clip]{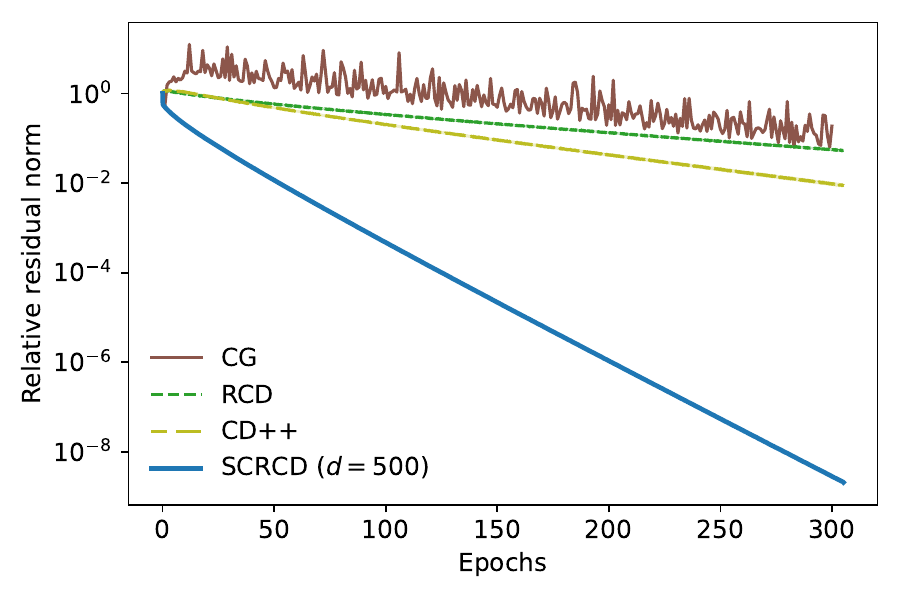}
    \includegraphics[width=0.495\linewidth, trim={0.2cm 0.4cm 0.2cm 0.3cm}, clip]{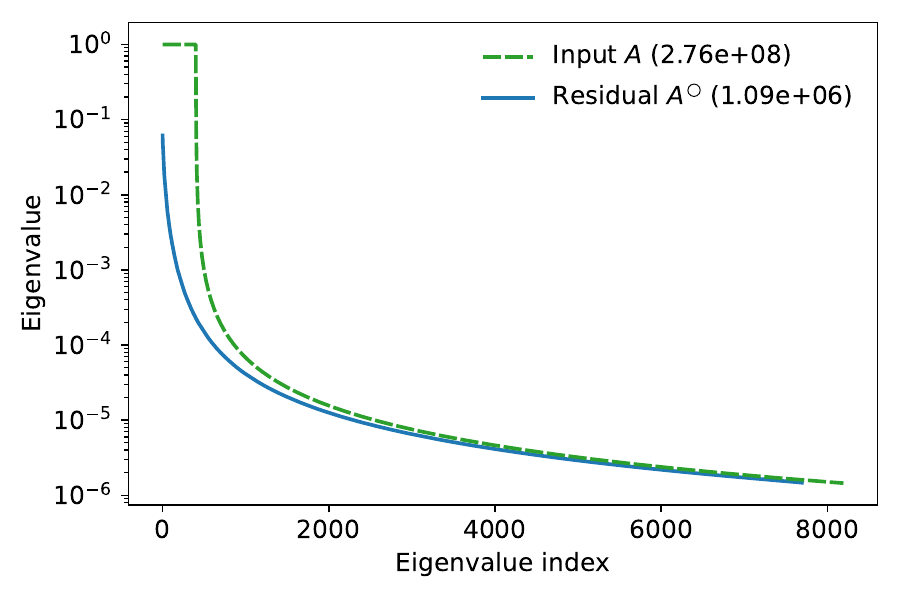}
    \vspace{-1.5\floatsep}
    \caption{
    Solving a synthetic $8,192 \times 8,192$ psd system $\mA \vx = \vb$ with approximate rank $r = 400$.
    \textbf{(Left)} Relative residual norm $\norm{\mA \vx^k - \vb}_2 / \norm{\mA \vx^0 - \vb}_2$ over 300 epochs for SC-RCD (with $d = 500$ and $\ell = 500$), as well as CG, RCD and CD++ (also with $\ell = 500$), using the same initial iterate as SC-RCD. Each epoch corresponds to a single pass over the entire dataset (i.e., one iteration of SC-RCD/RCD/CD++ corresponds to $\ell / n$ epochs).
    The lines depict the median over 100 independent runs with the same Nystr\"{o}m approximation.
    \textbf{(Right)} Eigenvalue spectra of $\mA$ and the residual matrix $\mA^{\circ}$ corresponding to the rank-$d$ approximation, with their condition numbers $\sum_i \lambda_i / \lambda_{\mathrm{min}}^+$ reported in brackets.
    }
    \label{fig:simlowrank}
    
    \vspace{0.5\floatsep}

    \includegraphics[width=0.495\linewidth, trim={0.2cm 0.4cm 0.2cm 0.3cm}, clip]{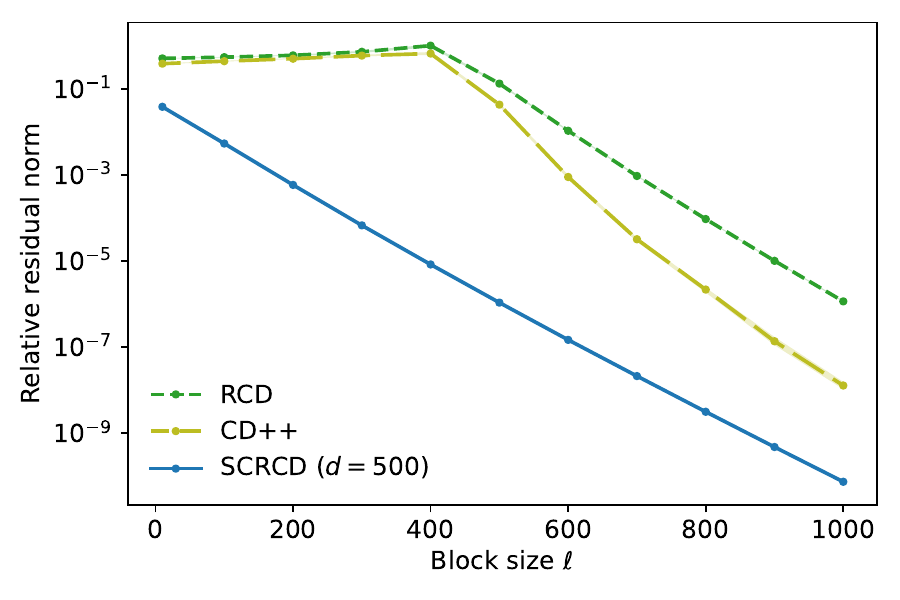}
    \includegraphics[width=0.495\linewidth, trim={0.2cm 0.4cm 0.2cm 0.3cm}, clip]{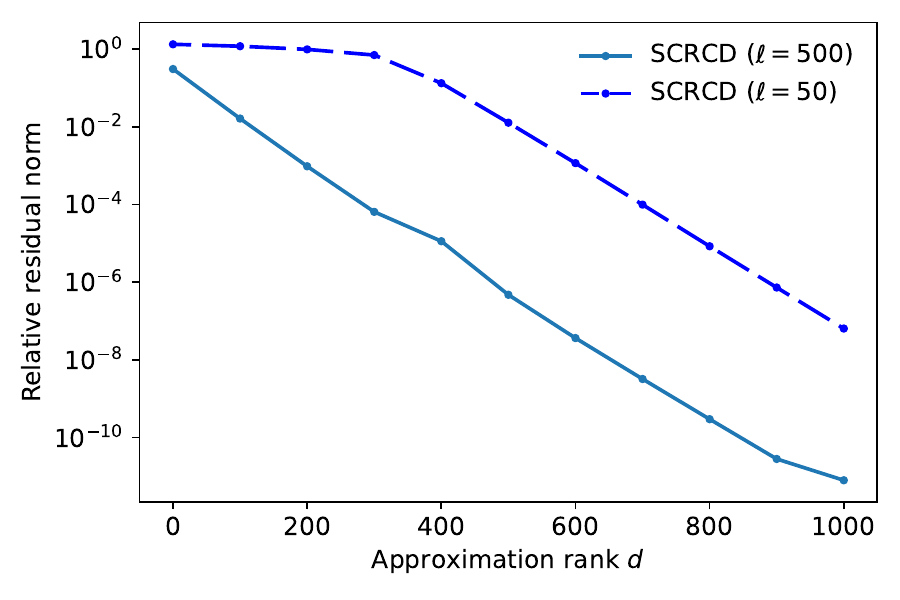}
    \vspace{-1.5\floatsep}
    \caption{
    Continuing the same setup as in Figure~\ref{fig:simlowrank}, the plots show the relative residual norm after 200 epochs for
    \textbf{(Left)} SC-RCD (with $d = 500$), RCD, and CD++ using various block sizes $\ell$; and
    \textbf{(Right)} SC-RCD (with $\ell \in \{ 50, 500 \}$) using various approximation ranks $d$.
    }
    \vspace{-0.5\floatsep}
    \label{fig:simlowrank_varydl}
\end{figure}

In the first experiment, we demonstrate the effectiveness of the SC-RCD method (Algorithm~\ref{alg:sc-rcd}) for solving approximately low-rank systems. We simulate a $n \times n$ psd linear system with $n = 8,192 = 2^{13}$, where the first $r = 400$ eigenvalues are equal to one (i.e., are ``large'' up to normalization) and subsequently decay as $\lambda_i = i^{-3/2}$ for $i > 400$, by defining a diagonal matrix $\mSigma$ with $\mSigma_{i,i} = \lambda_i$ and rotating with a uniformly random orthogonal matrix $\mU$ to form $\mA = \mU \mSigma \mU^{\tran}$. 

Figure~\ref{fig:simlowrank} (right) shows that the residual matrix $\mA^{\circ} = \mA - \mA \langle \mathcal{S} \rangle$ with approximation rank $d = 500 \approx 5.5 \sqrt{n}$ is much better conditioned than $\mA$.
Accordingly, Figure~\ref{fig:simlowrank} (left) shows that SC-RCD, using the corresponding rank-$d$ Nystr\"{o}m approximation and block size $\ell = 500$, converges effectively.
For comparison, we also show the convergence rate, measured on an epoch-basis, for related methods including the conjugate gradient method (CG); randomized coordinate descent (RCD) with blocks of size $\ell = 500$, sampled in the same way as SC-RCD; and the recently proposed CD++ method from Derezi{\'n}ski et al.~\cite{DerezinskiEtAl2025beyond}, which combines RCD with techniques such as adaptive acceleration and Hadamard preconditioning.

Figure~\ref{fig:simlowrank_varydl} shows the relative residual norm after 200 epochs for SC-RCD, RCD, and CD++ using various block sizes $\ell$ and approximation ranks $d$.
Figure~\ref{fig:simlowrank_varydl} (left) shows that the SC-RCD error decreases as $\ell$ increases. The theory that we develop (Theorem~\ref{thm:sc-rcd_convrate}) implies that this curve should be non-decreasing; however, the actual performance (significantly) exceeds this bound, which reflects how larger blocks are able to implicitly capture larger parts of the spectrum of $\mA$ (see~\cite{derezinski2024sharp, DerezinskiEtAl2025beyond} for related theory). It also shows that RCD and CD++ do not converge effectively until $\ell$ is large enough to implicitly capture the leading $r = 400$ eigenvalues, after which they significantly improve.
Figure~\ref{fig:simlowrank_varydl} (right) shows that the SC-RCD error also decreases as $d$ increases and RPCholesky computes a higher quality matrix approximation (Theorem~\ref{thm:rpcholesky_error}). We observe a significant improvement once $d$ is large enough, in combination with $\ell$ (due to the implicit effects of block size), to capture the $r$ large spectral outliers of $\mA$ in each iteration.

We note that while each iteration of SC-RCD incurs an additional computational cost of $O(nd)$ to enforce the subspace constraint compared to RCD, it is not the dominating term in the complexity of each iteration when $d, \ell \approx O(\sqrt{n})$ (see Section~\ref{sec:sc-rcd_implementation}). Correspondingly, the time to run SC-RCD and RCD in Figures~\ref{fig:simlowrank} and \ref{fig:simlowrank_varydl} with the same block size $\ell$ was found to be very close.
The iterations of CD++ with the same block size are somewhat faster due to its use of techniques such as approximate regularized projections and block memoization, which could be integrated with SC-RCD for practical efficiency as a part of future work.
However, we note that the theoretical guarantee~\eqref{eq:cd++_complexity} for CD++ requires the matrix $\mA$ to be stored in memory and preprocessed by a randomized Hadamard transform, or otherwise requires $\mA$ to possess some natural incoherence properties. This makes it more difficult to apply for large-scale problems, such as in the upcoming experiment.

\paragraph{KRR problem on real-life dataset with fast spectral decay.}

\begin{figure}[!htb]
    \centering
    \includegraphics[width=0.495\linewidth, trim={0.2cm 0.2cm 0.2cm 0.3cm}, clip]{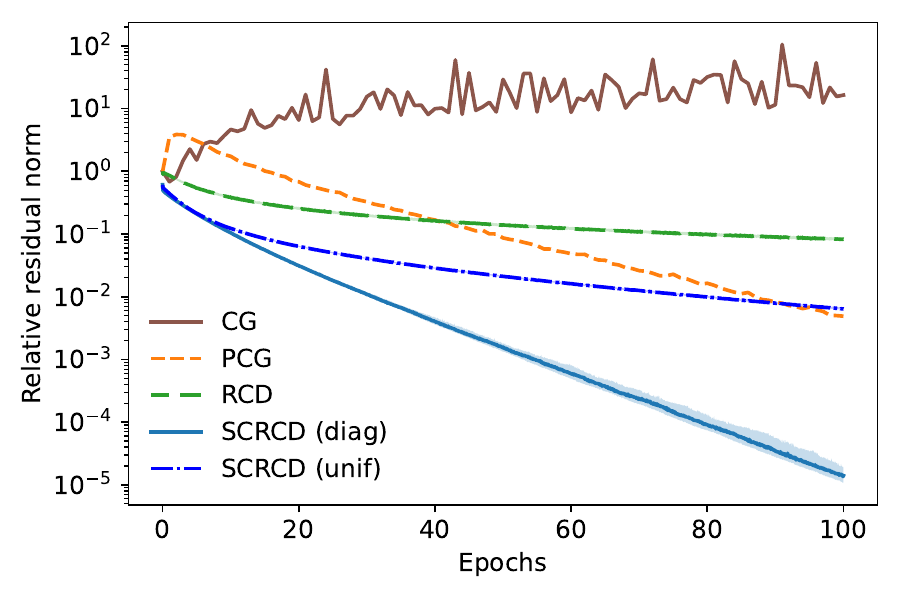}
    \includegraphics[width=0.495\linewidth, trim={0.2cm 0.2cm 0.2cm 0.3cm}, clip]{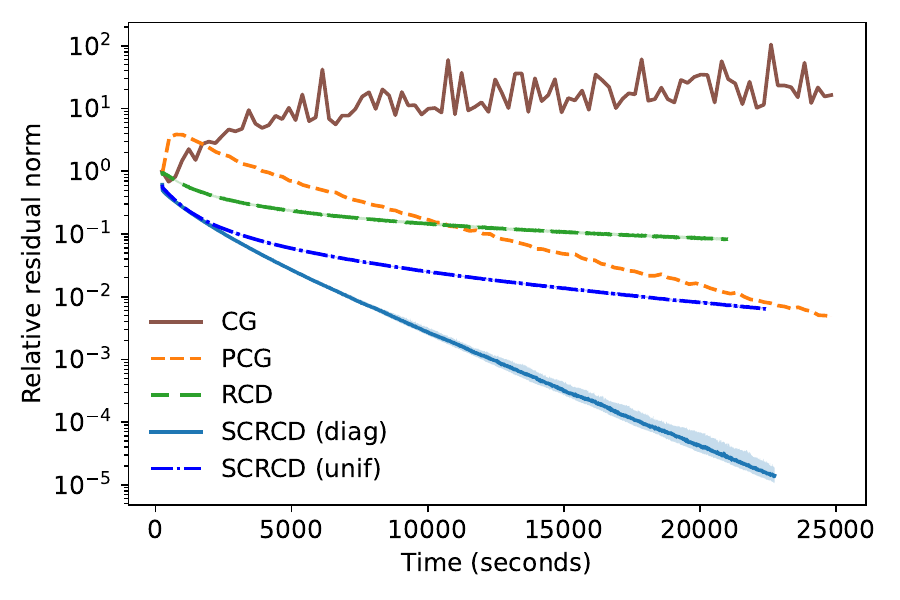}
    \includegraphics[width=0.495\linewidth, trim={0.2cm 0.2cm 0.3cm 0cm}, clip]{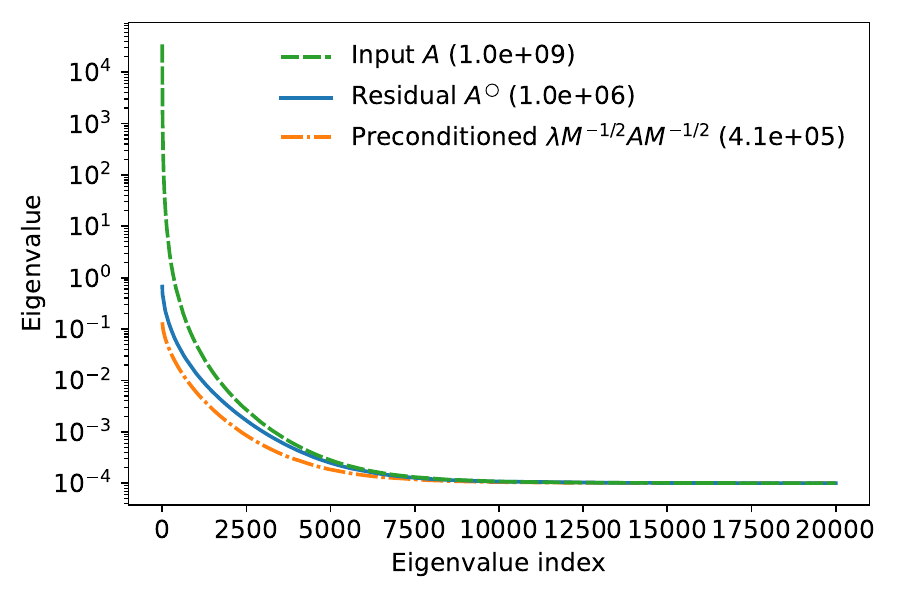}
    \includegraphics[width=0.495\linewidth, trim={0.2cm 0.2cm 0cm 0.3cm}, clip]{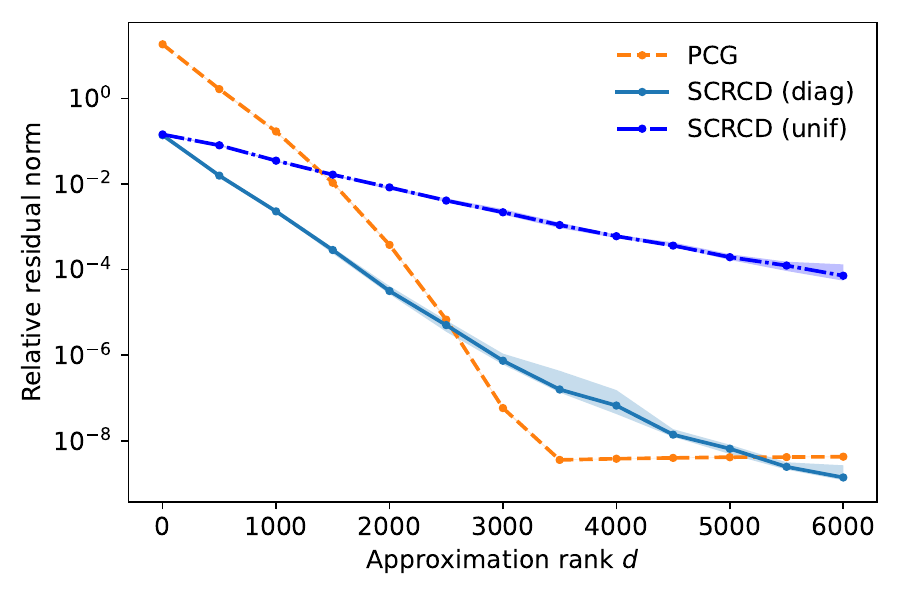}
    \vspace{-1.5\floatsep}
    \caption{
    Solving the KRR problem $(\mK + \lambda \mI) \vx = \vy$ on the \texttt{hls4ml\_lhc\_jets} dataset with $n = 100,000$ samples and a very small amount of regularization $\lambda = 10^{-9} n$.
    \textbf{(Top left)} Relative residual norm $\norm{(\mK + \lambda \mI) \vx^k - \vy}_2 / \norm{\vy}_2$ over 100 epochs for SC-RCD (with $d = 1,000$ and $\ell = 1,000$), as well as RCD (also with $\ell = 1,000$), CG, and PCG (also with $d = 1,000$).
    The lines (resp.\ shaded interval) depict the median (resp.\ 0.2- and 0.8-quantiles) over 100 independent runs.
    \textbf{(Top right)} Error in terms of time elapsed. The kernel matrix $\mK$ is not stored in memory, and entry evaluations represent the dominant computational cost.
    \textbf{(Bottom left)} The leading $20,000$ eigenvalues of $\mA = \mK + \lambda \mI$, the residual $\mA^{\circ}$, and the preconditioned $\lambda \mM^{-1/2} \mA \mM^{-1/2}$ corresponding to the rank-$d$ Nystr\"{o}m approximation, with their condition numbers $\sum_i \lambda_i / \lambda_{\mathrm{min}}^+$ reported in brackets.
    \textbf{(Bottom right)} The error after 50 epochs for SC-RCD and PCG with various approximation ranks $d$.}
    \vspace{-0.5\floatsep}
    \label{fig:lhc100}
\end{figure}

In the next experiment, we investigate the performance of the SC-RCD method for solving large-scale kernel ridge regression (KRR) problems where the matrix cannot be stored in memory, and the dominant computational cost comes from matrix evaluations.

To give a brief overview of KRR, suppose that we are given $n$ data points $\{ (\vz_i, y_i) \}_{i=1}^n$ with features $\vz_i \in \reals^p$ and response variable $y_i \in \reals$. A kernel matrix $\mK \in \reals^{n \times n}$ is formed with $\mK_{i,j} = \kappa(\vz_i, \vz_j)$ for some positive definite kernel function $\kappa: \reals^n \times \reals^n \to \reals$. For our experiments, we will use the Gaussian kernel with bandwidth parameter $\sigma > 0$, defined by
$
    \kappa(\vz_i, \vz_j) = \exp\left( -\norm{\vz_i - \vz_j}_2^2 / (2 \sigma^2) \right).
$
Then, given a regularization parameter $\lambda > 0$, the goal of KRR is to find a vector $\vx \in \reals^n$ to minimize $\norm{\mK \vx - \vy}_2^2 + \lambda \vx^{\tran} \mK \vx$, which is equivalent to solving the positive definite system $(\mK + \lambda \mI) \vx = \vy$.

In Figure~\ref{fig:lhc100}, we take $n = 100,000$ samples from the \texttt{hls4ml\_lhc\_jets} dataset~\cite{hls4ml_lhc_jet_dataset}, which consists of features $\vz_i \in \reals^{16}$ for predicting jet classes from LHC proton-proton collisions. We consider solving $(\mK + \lambda \mI) \vx = \vy$ using the Gaussian kernel with bandwidth $\sigma = 3$ and a small regularization parameter $\lambda = 10^{-9} n$, which results in a more ill-conditioned and challenging system to solve.
Figure~\ref{fig:lhc100} (bottom) shows that the eigenvalues of $\mK$ decay exponentially, so the regularized system has a flat-tailed spectrum that quickly decays to $\lambda$.

We consider the SC-RCD method where the blocks consist of indices sampled with weights proportional to the diagonal of the residual matrix or uniformly, as well as RCD (where both forms of sampling are equivalent since $\mK$ has unit diagonals). In each iteration, we perform inexact projections, where $\valpha$ is solved up to a relative error of $0.05$ using CG with a simple Jacobi preconditioner (i.e., diagonal normalization). For comparison, we also solve the system using the preconditioned CG method (PCG) proposed by D{\'{i}}az et al.~\cite{DiazEtAl2023robust}, which uses a preconditioner $\mM = \mF \mF^{\tran} + \lambda \mI$, constructed from an approximation $\widehat{\mK} = \mF \mF^{\tran}$ of $\mK$ using RPCholesky.

Figure~\ref{fig:lhc100} (top) show the convergence rate of these iterative solvers, measured in terms of the number of epochs completed (left) and the total time elapsed (right). We observe that the SC-RCD method with a relatively small Nystr\"{o}m approximation significantly improves upon RCD, analogous to the improvement of PCG over CG as shown by~\cite{DiazEtAl2023robust}. 
Furthermore, SC-RCD converges faster than PCG with the same approximation rank $d = 1,000$ used. However, Figure~\ref{fig:lhc100} (bottom right) shows that the improvement in the rate of PCG with larger $d$ is faster than for SC-RCD. In practice, the choice of $d$ may be limited by the availability of memory.

This experiment provides limited evidence of how coordinate descent-based methods can be competitive with methods such as preconditioned CG for large-scale problems where entry evaluations are costly. 
There may be further computational advantages of CD-based methods, such as the possibility for acceleration~\cite{TuEtAl2017locality} and parallelization (e.g., averaging over mini-batches~\cite[Algorithm~2]{RichtarikTakac2020}), but we do not investigate these possibilities.

\paragraph{KRR problem with slower spectral decay.}

\begin{figure}[!htb]
    \centering
    \includegraphics[width=0.495\linewidth, trim={0.2cm 0.4cm 0.2cm 0.3cm}, clip]{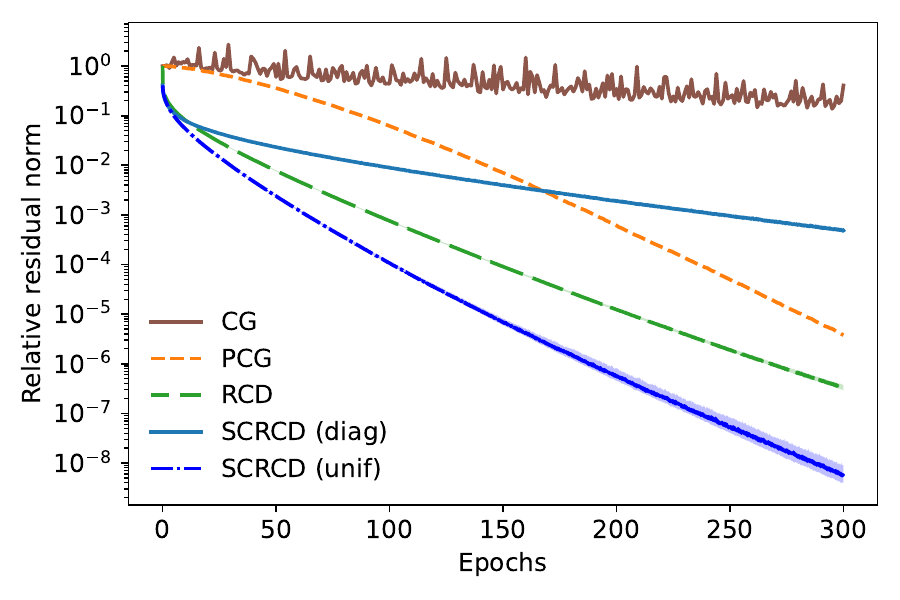}
    \includegraphics[width=0.495\linewidth, trim={0.2cm 0.4cm 0.2cm 0.3cm}, clip]{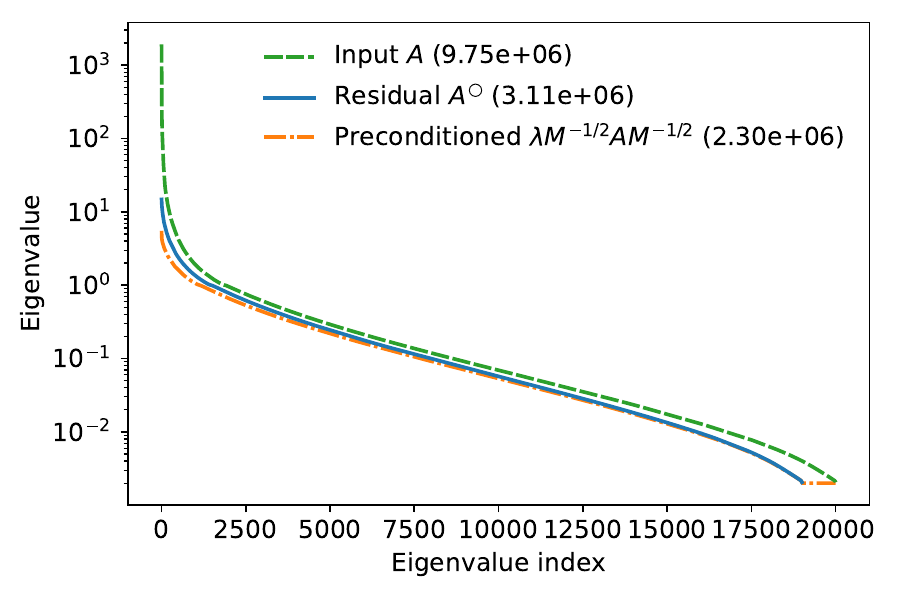}
    \vspace{-1.5\floatsep}
    \caption{
    Solving the KRR problem $(\mK + \lambda \mI) \vx = \vy$ on the \texttt{sensorless} dataset with $n = 20,000$ samples and a small regularization parameter $\lambda = 10^{-7} n$.
    \textbf{(Left)} Relative residual norm over 300 epochs for SC-RCD (with $d = 1,000$ and $\ell = 1,000$), RCD (also with $\ell = 1,000$), CG, and PCG (also with $d = 1,000$).
    \textbf{(Right)} Eigenvalue spectra of $\mA = \mK + \lambda \mI$, the residual $\mA^{\circ}$, and the preconditioned $\lambda \mM^{-1/2} \mA \mM^{-1/2}$ corresponding to the rank-$d$ Nystr\"{o}m approximation, with their condition numbers $\sum_i \lambda_i / \lambda_{\mathrm{min}}^+$ reported in brackets. 
    }
    \vspace{-0.5\floatsep}
    \label{fig:sensorless20}
\end{figure}

In the final experiment, we investigate the performance of SC-RCD for solving another KRR problem on a dataset with slower spectral decay. We also demonstrate that how the blocks are sampled can play a critical role in the convergence rate of SC-RCD.

Specifically, we consider solving $(\mK + \lambda \mI) \vx = \vy$ using $n = 20,000$ samples from the \texttt{sensorless} dataset~\cite{LIBSVM}, which consists of features $\vz_i \in \reals^{48}$, the Gaussian kernel with bandwidth $\sigma = 3$, and a small regularization parameter $\lambda = 10^{-7} n$.
This dataset was identified as one of the more difficult KRR problems studied in~\cite{DiazEtAl2023robust}.
Figure~\ref{fig:sensorless20} (right) confirms that the kernel matrix exhibits much slower spectral decay, making it far more difficult to find a good low-rank approximation.
Figure~\ref{fig:sensorless20} (left) shows that SC-RCD with uniformly sampled blocks exhibits the fastest convergence rate, and diagonal sampling---which has been the most effective for systems with rapid spectral decay so far---actually performs poorly.
We observe that SC-RCD with uniform sampling improves upon RCD, as expected from Proposition~\ref{prop:sc-rcd_convrate_unif}.

\section{Conclusion}

We proposed and analyzed the SC-RCD method for solving psd linear systems $\mA \vx = \vb$, which combines the classical randomized block coordinate descent algorithm with a rank-$d$ matrix approximation, efficiently computable using an algorithm such as RPCholesky.
We proved that it is a lightweight algorithm that can obtain an $\epsilon$-relative error solution using $O(nd)$ memory and $O\bigl( (n^2 + nd^2) \cdot \bar{\kappa}_r(\mA) \log(1/\epsilon) \bigr)$ arithmetic operations, where $\bar{\kappa}_r(\mA) = \sum_{i > r} \lambda_i(\mA) / \lambda_{\mathrm{min}}^+(\mA)$ is the normalized tail condition number of $\mA$ and $r$ is typically close to the approximation rank $d$. This makes SC-RCD effective for solving large-scale, dense systems with rapid spectral decay, such as those arising in kernel ridge regression. We presented numerical experiments in support of these results.

Some directions for future work include combining the subspace-constrained framework with other computational techniques, such as those employed in~\cite{DerezinskiEtAl2025beyond}, for further practical efficiency. For example, momentum-based acceleration would  help attain an improved complexity in terms of the condition number dependence.
A more general question suggested by this work is how one can efficiently learn subspaces that control different parts of the spectrum, such as small spectral outliers. 
Another future direction is to investigate how constraining the dynamics within a selected subspace can be used to accelerate other iterative algorithms based on the sketch-and-project approach, particularly those addressing more general nonlinear problems.

\section*{Acknowledgments}

We would like to thank Ethan Epperly and Robert Webber for helpful suggestions and discussions.
We would also like to thank the anonymous reviewers for their comments and suggestions that have improved the presentation of the paper.

\phantomsection
\addcontentsline{toc}{section}{References}
\section*{References}
\printbibliography[heading=none]

\begin{appendix}

\section{Subspace-constrained sketch-and-project technical proofs} \label{sec:sc-sap_proofs}

In this section, we give the technical proofs of Lemmas~\ref{lem:update} and~\ref{lem:error_invariance} for subspace-constrained sketch-and-project from Section~\ref{sec:sc-sap}.

\begin{proof}[Proof of Lemma~\ref{lem:update}]
Given the iterate $\vx^k$ after the $k$th iteration, define $\bar{\vz} := \vx^k - \vx^{k+1}$. Then from~\eqref{eq:sc_sap}, together with the change of variables $\vw = \mB^{1/2} (\vx^k - \vx)$, we have
\begin{equation} \label{eq:opt-problem}
    \mB^{1/2} \bar{\vz} =
    \left[
    \begin{tabular}{rl}
        $\displaystyle\argmin_{\vw \in \reals^n}$ & $\vw^{\tran} \vw$ \\
        \text{such that} & $\mS \mA \mB^{-1/2} \vw = \mS(\mA \vx^k - \vb)$, \\
                         & $\mQ \mA \mB^{-1/2} \vw = \mQ(\mA \vx^k - \vb)$.
    \end{tabular}
    \right]
\end{equation}
Note that since $\mQ \mA \vx^k = \mQ \vb$, the second constraint is equivalent to $\vw \in \nullspace(\mQ \mA \mB^{-1/2})$, or $\vw = \mP \vw$.
By introducing Lagrange multipliers $\vlambda \in \reals^{\ell}$, $\vtau \in \reals^{d}$, we deduce that the optimal $\vw_{\ast} = \mB^{1/2} \bar{\vz}$ solves the first-order conditions
\begin{align*}
    \vw_{\ast}  + \mB^{-1/2} \mA^{\tran} \mS^{\tran} \vlambda + \mB^{-1/2} \mA^{\tran} \mQ^{\tran} \vtau &= \vzero, \\
    \mS \mA \mB^{-1/2} \vw_{\ast}  &= \mS(\mA \vx^k - \vb), \\
    \mP \vw_{\ast}  &= \vw_{\ast}.
\end{align*}
First, by definition of $\mP$, we have $\mP \mB^{-1/2} \mA^{\tran} \mQ^{\tran} = \vzero$.
Hence, by multiplying the first equation by $\mP$, we obtain $\mP \vw_{\ast} + \mP \mB^{-1/2} \mA^{\tran} \mS^{\tran} \vlambda = \vzero$. By combining this with $\mP \vw_{\ast} = \vw_{\ast}$, we deduce that
\[
    \vw_{\ast} = -\mP \mB^{-1/2} \mA^{\tran} \mS^{\tran} \vlambda.
\]
Thus, $\vw_{\ast} \in \range(\mP \mB^{-1/2} \mA^{\tran} \mS^{\tran})$, which implies that $\mZ \vw_{\ast} = \vw_{\ast}$. This shows that
\begin{align} \label{eq:w-update}
    \vw_{\ast} = \mZ \vw_{\ast} &= \mP \mB^{-1/2} \mA^{\tran} \mS^{\tran} (\mS \mA \mB^{-1/2} \mP \mB^{-1/2} \mA^{\tran} \mS^{\tran})^{\dagger} \mS \mA \mB^{-1/2} \mP \vw_{\ast} \nonumber\\
    &= \mP \mB^{-1/2} \mA^{\tran} \mS^{\tran} (\mS \mA \mB^{-1/2} \mP \mB^{-1/2} \mA^{\tran} \mS^{\tran})^{\dagger} \mS(\mA \vx^k - \vb),
\end{align}
where in the last line we used that $\mP \vw_{\ast} = \vw_{\ast}$, and so $\mS \mA \mB^{-1/2} \mP\vw_{\ast} = \mS(\mA \vx^k - \vb)$. Recalling that $\vx^{k+1} = \vx^k - \mB^{-1/2} \vw_{\ast}$, the update rule~\eqref{eq:lem_update_1} follows from~\eqref{eq:w-update}.

Next, from the update rule~\eqref{eq:lem_update_1}, we obtain
\begin{align*}
    \mB^{1/2}(\vx^{k+1} - \vx^*) &= \mB^{1/2} (\vx^k - \vx^*) - \mP \mB^{-1/2} \mA^{\tran} \mS^{\tran} (\mS \mA \mB^{-1/2} \mP \mB^{-1/2} \mA^{\tran} \mS^{\tran})^{\dagger} \mS \mA(\vx^k - \vx^*) \\
    &= \mB^{1/2} (\vx^k - \vx^*) - \mZ \mB^{1/2}(\vx^k - \vx^*),
\end{align*}
where in the last line we used $\vx^k - \vx^* = \mB^{-1/2} \mP \mB^{1/2}(\vx^k - \vx^*)$.
Indeed, since $\mQ \mA(\vx^k - \vx^*) = \mQ(\mA \vx^k - \vb) = \vzero$, we have $\mB^{1/2} (\vx^k - \vx^*) \in \nullspace(\mQ \mA \mB^{-1/2})$ and so $\mP \mB^{1/2}(\vx^k - \vx^*) = \mB^{1/2} (\vx^k - \vx^*)$.
This concludes the proof of the fixed point iteration~\eqref{eq:lem_update_2}.
\end{proof}

\begin{proof}[Proof of Lemma~\ref{lem:error_invariance}]
From the formula~\eqref{eq:xstar_formula} for $\vx^*$, we see that $\mB^{1/2}(\vx^0 - \vx^*) \in \range(\mB^{-1/2} \mA^{\tran})$. Furthermore, since the initial iterate solves $\mQ \mA \vx^0 = \mQ \vb$, we have $\mQ \mA(\vx^0 - \vx^*) = \mQ(\mA \vx^0 - \vb) = \vzero$, and thus $\mB^{1/2}(\vx^0 - \vx^*) \in \nullspace(\mQ \mA \mB^{-1/2}) = \range(\mP)$. Hence,
\[
    \mB^{1/2}(\vx^0 - \vx^*) = \mP \mB^{1/2}(\vx^0 - \vx^*) \in \range(\mP \mB^{-1/2} \mA^{\tran}).
\]
Next, observe that $\mB^{1/2}(\vx^k - \vx^*) \in \range(\mP)$ for all $k \geq 0$ since the subsequent iterates $\vx^k$ continue to solve $\mQ \mA \vx^k = \mQ \vb$.
From the fixed point iteration~\eqref{eq:lem_update_2} in Lemma~\ref{lem:update}, we have $\mB^{1/2}(\vx^k - \vx^*) = \mB^{1/2}(\vx^{k-1} - \vx^*) - \mZ \mB^{1/2}(\vx^{k-1} - \vx^*)$.
Since $\mZ$ is the orthogonal projector onto $\range(\mP \mB^{-1/2} \mA^{\tran} \mS^{\tran})$, it follows from induction that $\mB^{1/2} (\vx^k - \vx^*) \in \range(\mP \mB^{-1/2} \mA^{\tran})$ for all $k \geq 0$.
\end{proof}

\section{Extension of SC-RCD for least-squares problems} \label{sec:least_squares}

\begin{algorithm}[!htb]
\caption{SC-RCD: least squares} \label{alg:sc-rcd_ls}
\begin{algorithmic}[1]
    \Require{Matrix $\mA \in \reals^{m \times n}$, vector $\vb \in \reals^{m}$, approximation rank $d$, block size $\ell$}
    \Ensure{Approximate solution $\vx \in \reals^n$ of $\argmin_{\vx} \norm{\mA \vx - \vb}_2$, residual vector $\vr = \mA \vx - \vb \in \reals^m$}
    \State Compute pivot set $\mathcal{S} \subseteq [n]$ and $\mQ \in \reals^{m \times d}$, $\mR \in \reals^{d \times n}$ defining column-pivoted partial QR decomp. $\widehat{\mA} = \mQ \mR$, and set $\mA^{\circ} \leftarrow \mA - \widehat{\mA}$  \Comment{E.g., ~\cite[Alg.~7]{ChenEtAl2024}}
    \State Compute $\mD \leftarrow (\mR_{:,\mathcal{S}})^{-1} \mQ^{\tran} \in \reals^{d \times m}$ and $\mC \leftarrow \mD \mA \in \reals^{d \times n}$
    \State Set $\vx \leftarrow \mD \vb \in \reals^n$ and $\vr \leftarrow \mA \vx - \vb \in \reals^m$
    \State Set $\vp \leftarrow \vzero_{n \times 1}$, and compute $\vp_j \leftarrow \norm{\mA^{\circ}_{:,j}}_2^2 / \norm{\mA^{\circ}}_F^2$ for $j = [n] \setminus \mathcal{S}$
    \For{$k = 1, 2, \ldots$}
        \State Sample subset $\mathcal{J} = \{ j_1, \ldots, j_\ell \}$ of $\ell$ columns with $j_1, \ldots, j_\ell \sim \vp$ i.i.d.  \label{alg:sc-rcd_ls_sampling_step}
        \State Solve $(\mA^{\circ}_{:, \mathcal{J}})^{\tran} \mA^{\circ}_{:, \mathcal{J}} \valpha = (\mA^{\circ}_{:, \mathcal{J}})^{\tran} \vr$ for $\valpha \in \reals^\ell$
        \State $\vbeta \leftarrow \mC_{:, \mathcal{J}} \valpha \in \reals^d$
        \State $\vx_{\mathcal{J}} \leftarrow \vx_{\mathcal{J}} - \valpha$, $\vx_{\mathcal{S}} \leftarrow \vx_{\mathcal{S}} + \vbeta$
        \State $\vr \leftarrow \vr - \mA^{\circ}_{:, \mathcal{J}} \valpha$
    \EndFor
\end{algorithmic}
\end{algorithm}

In this section, we will briefly explain how the SC-RCD method can be adapted to solving the least-squares problem
\begin{equation} \label{eq:ls}
    \argmin_{\vx \in \reals^n} \norm{\mA \vx - \vb}_2, \quad \text{where } \mA \in \reals^{m \times n}.
\end{equation}
Since~\eqref{eq:ls} reduces to the solution of the normal equations $\mA^{\tran} \mA \vx = \mA^{\tran} \vb$, this can be solved by applying the psd SC-RCD method on the normal equations. In the following, we will show that the algorithm can be implemented without explicitly forming the psd matrix $\mA^{\tran} \mA$ as a column-action method; see Algorithm~\ref{alg:sc-rcd_ls} for pseudocode.

In this setting, the natural analogue of the Nystr\"{o}m approximation is the \emph{column projection approximation} $\mPi_{\mA, \mathcal{S}} \mA$ of $\mA$, where $\mPi_{\mA, \mathcal{S}}$ is the orthogonal projector onto the span of the columns of $\mA$ indexed by $\mathcal{S}$.
The key observation is that the entries of the Gram matrix $(\mA_{:,\mathcal{S}})^{\tran} \mA_{:,\mathcal{S}}$ give the inner products between the columns of $\mA$ indexed by $\mathcal{S}$; in particular, the squared column norms can be read off the diagonal (see, e.g.,~\cite[\S 3]{ChenEtAl2024} for more details on this classical connection).

Indeed, the RPCholesky algorithm can be naturally adapted to a \emph{randomly pivoted QR} algorithm that outputs a column projection approximation in the form of a partial QR decomposition $\widehat{\mA} = \mQ \mR$ of $\mA$, where $\mQ \in \reals^{m \times d}$ has orthonormal columns and $\mR \in \reals^{d \times n}$ is upper triangular (after pivoting to bring the columns in $\mathcal{S}$ to the front), such that $\widehat{\mA}_{:,\mathcal{S}} = \mA_{:,\mathcal{S}}$ and $\range(\widehat{\mA}) = \range(\mA_{:,\mathcal{S}})$; see \cite[Algorithm~7]{ChenEtAl2024}.

\subsubsection*{Derivation of Algorithm~\ref{alg:sc-rcd_ls}}
Suppose that we are given a column projection approximation $\mPi_{\mA, \mathcal{S}} \mA$ of $\mA$ with $d$ pivots $\mathcal{S} \subseteq [n]$ and an initial iterate $\vx^0$ satisfying $(\mA^{\tran} \mA)_{\mathcal{S},:} \vx^0 = (\mA^{\tran} \vb)_{\mathcal{S}}$, or equivalently $(\mA_{:,\mathcal{S}})^{\tran} (\mA \vx^0 - \vb) = \vzero$.
Let
\begin{equation} \label{eq:column_approx}
    \widehat{\mA} := \mPi_{\mA, \mathcal{S}} \mA
    \quad\text{and}\quad
    \mA^{\circ} := \mA - \widehat{\mA}
\end{equation}
be the column projection approximation of $\mA$ with respect to the columns indexed by $\mathcal{S}$ and the corresponding residual matrix respectively.
Observe that if $$\mP = \mI - (\mA^{\tran} \mA)^{1/2} \ve_{\mathcal{S}} ((\mA^{\tran} \mA)_{\mathcal{S},\mathcal{S}})^{\dagger} \ve_{\mathcal{S}}^{\tran} (\mA^{\tran} \mA)^{1/2}$$ is the orthogonal projector onto $\nullspace(\ve_{\mathcal{S}}^{\tran} (\mA^{\tran} \mA)^{1/2})$, then
\begin{equation} \label{eq:column_approx_proj}
    (\mA^{\tran} \mA)^{1/2} \mP (\mA^{\tran} \mA)^{1/2}
    = \mA^{\tran} (\mI - \mPi_{\mA, \mathcal{S}}) \mA
    = (\mA^{\circ})^{\tran} \mA^{\circ}.
\end{equation}
Hence, after some algebraic manipulations, the update~\eqref{eq:sc-rcd_update} for the iterate $\vx^k$ for solving the psd system $\mA^{\tran} \mA \vx = \mA^{\tran} \vb$ with SC-RCD is equivalent to the following:
\begin{equation} \label{eq:sc-rcd_ls_update}
    \vx^{k+1}
    = \vx^k - \ve_{\mathcal{J}} \valpha^k + \ve_{\mathcal{S}} \vbeta^k,
\end{equation}
where
\[
    \valpha^k := (((\mA^{\circ})^{\tran} \mA^{\circ})_{\mathcal{J}, \mathcal{J}})^{\dagger} (\mA_{:, \mathcal{J}})^{\tran} \vr^k
    \quad\text{and}\quad
    \vbeta^k := \mC_{:, \mathcal{J}} \valpha^k,
\]
with $\mC := ((\mA^{\tran} \mA)_{\mathcal{S},\mathcal{S}})^{\dagger} (\mA^{\tran} \mA)_{\mathcal{S},:}$ and $\vr^k = \mA \vx^k - \vb$.
Note that the subspace constraint maintains the invariant $(\mA_{:,\mathcal{S}})^{\tran} (\mA \vx^k - \vb) = \vzero$. Therefore, $(\widehat{\mA}_{:,\mathcal{J}})^{\tran} \vr^k = (\mA_{:,\mathcal{J}})^{\tran} \mPi_{\mA, \mathcal{S}} \vr^k = \vzero$, and we can replace $(\mA_{:, \mathcal{J}})^{\tran} \vr^k$ with $(\mA^{\circ}_{:, \mathcal{J}})^{\tran} \vr^k$ (i.e., $\valpha^k$ is the solution of a highly overdetermined least squares problem $\argmin_{\valpha} \norm{\mA^{\circ}_{:,\mathcal{J}} \valpha - \vr^k}_2$).
Furthermore, the update~\eqref{eq:sc-rcd_update_resid} for the residual vector $\vr^k$ is equivalent to
\begin{align} \label{eq:sc-rcd_ls_update_resid}
    \vr^{k+1}
    &= \vr^{k} - (\mI - \mPi_{\mA, \mathcal{S}}) \mA_{:, \mathcal{J}} \valpha^k
    = \vr^{k} - \mA^{\circ}_{:, \mathcal{J}} \valpha^k.
\end{align}
The updates~\eqref{eq:sc-rcd_ls_update} and~\eqref{eq:sc-rcd_ls_update_resid} are summarized in the pseudocode in Algorithm~\ref{alg:sc-rcd_ls}.
Each iteration requires accessing the columns of $\mA$ and $\widehat{\mA}$ indexed by the sampled block $\mathcal{J}$.
Similar to the psd case (Section~\ref{sec:sc-rcd_implementation}), the partial QR structure of the approximation $\widehat{\mA}$ can be used to compute the auxiliary matrix $\mC = (\mA_{:,\mathcal{S}})^{\dagger} \mA$ and a valid initialization $\vx^0 = (\mA_{:,\mathcal{S}})^{\dagger} \vb$ more efficiently by observing that $(\mA_{:,\mathcal{S}})^{\dagger} = (\mQ \mR_{:,\mathcal{S}})^{\dagger} = (\mR_{:,\mathcal{S}})^{-1} \mQ^{\tran}$ if $\mA_{:,\mathcal{S}}$ has full rank, recalling that $\mR_{:,\mathcal{S}}$ is upper triangular.

The following result states the convergence rate of the least squares SC-RCD method (Algorithm~\ref{alg:sc-rcd_ls}), which immediately follows from Theorem~\ref{thm:sc-rcd_convrate} for a fixed pivot set $\mathcal{S}$. As in the psd SC-RCD case, this can be further combined with bounds for the quality of the low-rank approximation $\widehat{\mA}$ such as~\cite[Corollary~5.2]{ChenEtAl2024}, which we do not elaborate on.

\begin{theorem} \label{thm:sc-rcd_ls_convrate}
Let $\mA \in \reals^{m \times n}$ and $\vx^*$ be any solution of the least squares problem~\eqref{eq:ls}. Suppose that $\{ \vx^k \}_{k \geq 0}$ are the iterates defined by~\eqref{eq:sc-rcd_ls_update} with a fixed subset $\mathcal{S} \subseteq [n]$, and the block $\mathcal{J} = \{ j_1, \ldots, j_\ell \}$ in each iteration consists of $\ell$ columns independently sampled according to the distribution $\{ \norm{\mA^{\circ}_{:,j}}_2^2 / \norm{\mA^{\circ}}_F^2 \}_{j=1}^n$. Then
\[
    \E \norm{\vx^k - \vx^*}_{\mA^{\tran} \mA}^2 \leq \left( 1 - \frac{\sigma_{\mathrm{min}}^+(\mA^{\circ})^2}{\norm{\mA^{\circ}}_F^2} \right)^{k \ell} \cdot \norm{\vx^0 - \vx^*}_{\mA^{\tran} \mA}^2,
\]
where $\mA^{\circ} = \mA - \mPi_{\mA, \mathcal{S}} \mA$, and $\sigma_{\mathrm{min}}^+(\mA^{\circ})$ is the smallest non-zero singular value of $\mA^{\circ}$.
Note that $\norm{\vx^k - \vx^*}_{\mA^{\tran} \mA}^2 = \norm{\mA \vx^k - \vb}_2^2 - \norm{\mA \vx^* - \vb}_2^2$ measures the suboptimality in the least squares objective.
\end{theorem}

\section{Additional numerical experiments} \label{sec:additional_numexp}

In this section, we present additional experiments to corroborate the findings reported in Section~\ref{sec:numerical_results}. We adopt a similar KRR setup as in Figures~\ref{fig:lhc100} and \ref{fig:sensorless20}: we take $n = 20,000$ samples $(\vz_i, y_i) \in \reals^p$ from a selection of datasets considered in~\cite[Table~1]{DiazEtAl2023robust}, which are sourced from OpenML~\cite{OpenML} and LibSVM~\cite{LIBSVM}, and solve $(\mK + \lambda \mI) \vx = \vy$ using the Gaussian kernel $\mK$ with bandwidth $\sigma = 3$ and a small regularization parameter $\lambda = 10^{-8} n$.
We use SC-RCD (with $d = 1,000 \approx 7 \sqrt{n}$ and $\ell = 1,000$), RCD (with $\ell = 1,000$), CG, and PCG (with $d = 1,000$).

Figures~\ref{fig:app_numexp_rapid} and~\ref{fig:app_numexp_slower} report the convergence trajectories, showing the relative residual norm $\norm{(\mK + \lambda \mI) \vx^k - \vy}_2 / \norm{\vy}_2$ over the first 300 epochs (the median and 0.2/0.8-quantiles over 10 independent runs are reported), and the corresponding eigenvalue spectra for eight datasets, loosely grouped in terms of whether the kernel matrix exhibits rapid or slower spectral decay.

\begin{figure}[!htb]
    \centering
    \includegraphics[width=0.495\linewidth, trim={0 0.85cm 0.1cm 0.25cm}, clip]{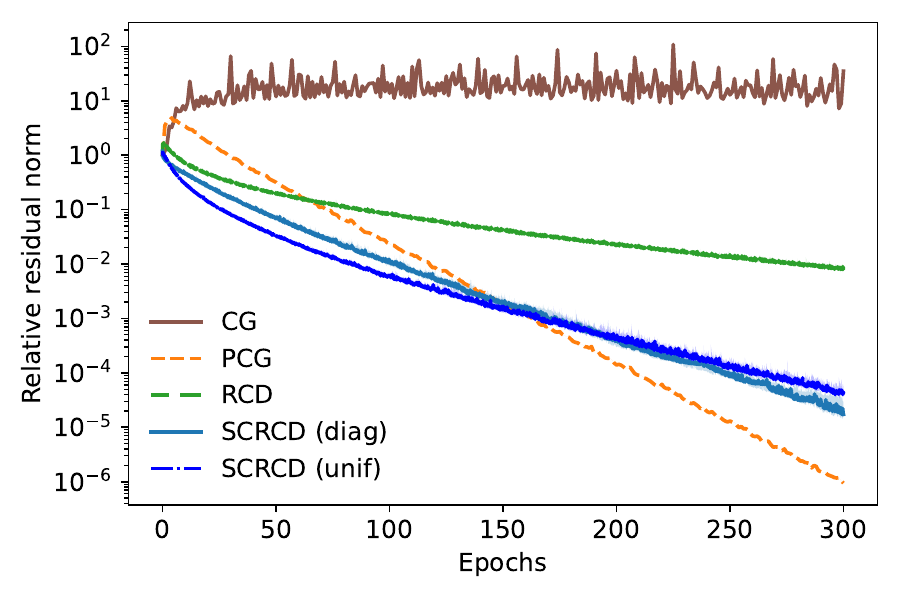}
    \includegraphics[width=0.495\linewidth, trim={0.2cm 1cm 0.3cm 0}, clip]{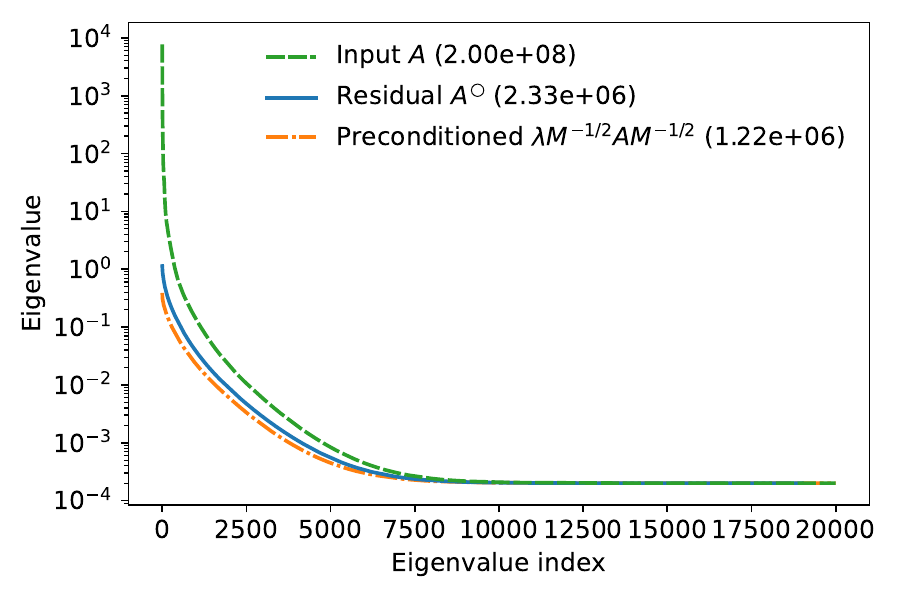}
    \includegraphics[width=0.495\linewidth, trim={0.2cm 1cm 0.1cm 0}, clip]{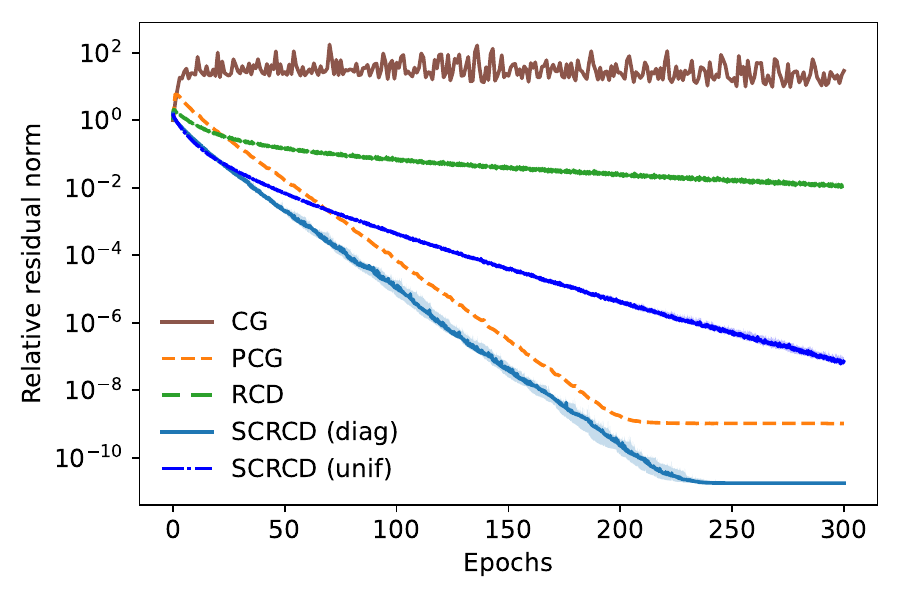}
    \includegraphics[width=0.495\linewidth, trim={0.2cm 1cm 0.3cm 0}, clip]{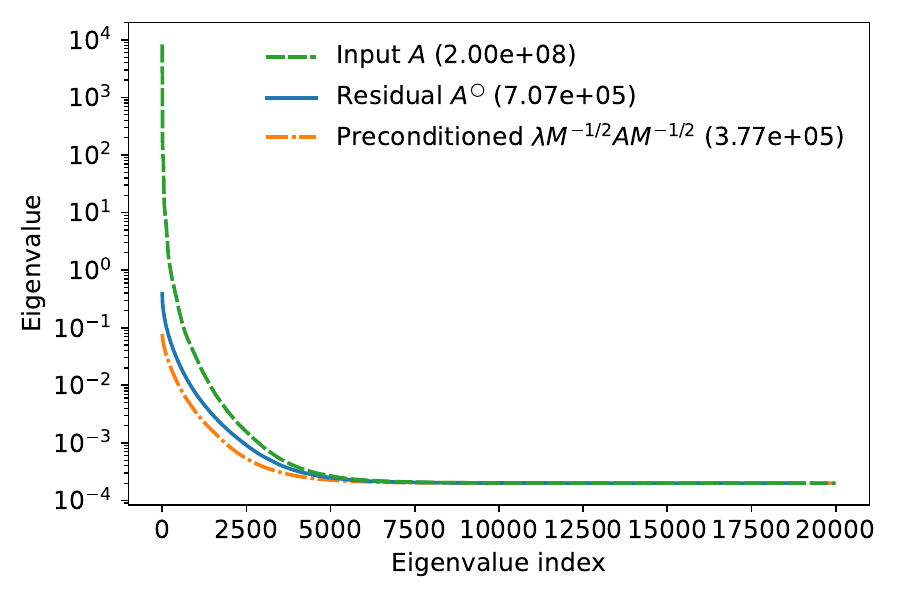}
    \includegraphics[width=0.495\linewidth, trim={0.2cm 1cm 0.3cm 0}, clip]{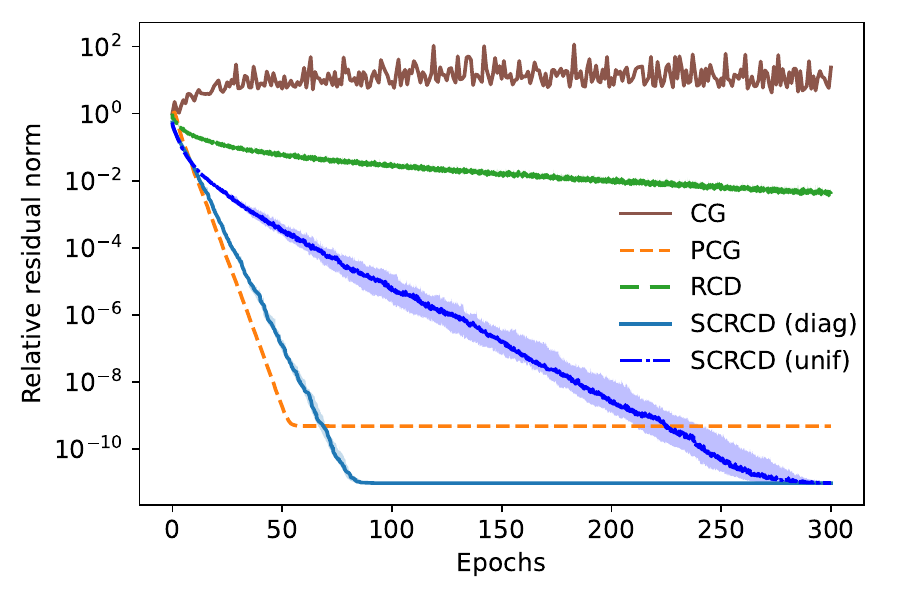}
    \includegraphics[width=0.495\linewidth, trim={0.2cm 1cm 0.3cm 0}, clip]{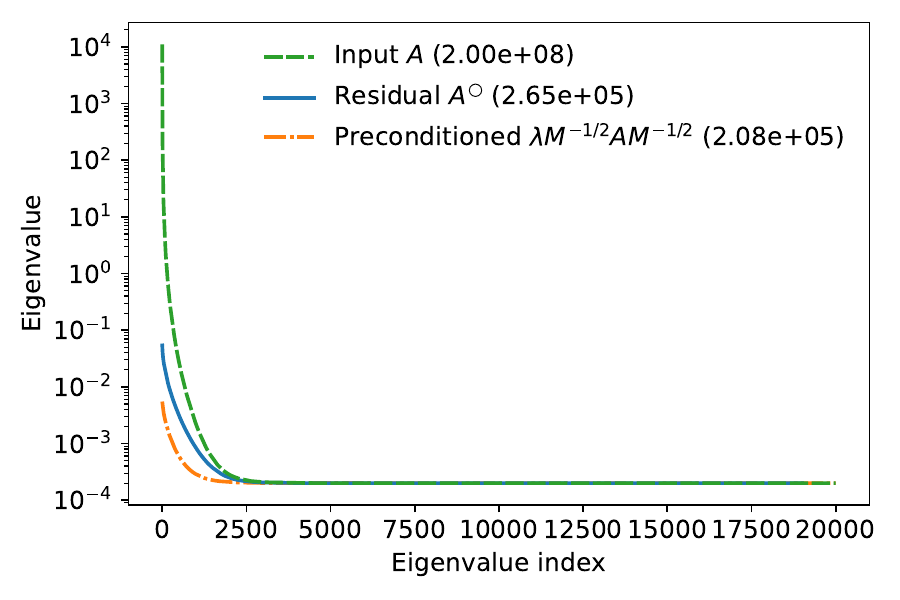}
    \includegraphics[width=0.495\linewidth, trim={0.2cm 0.3cm 0.3cm 0}, clip]{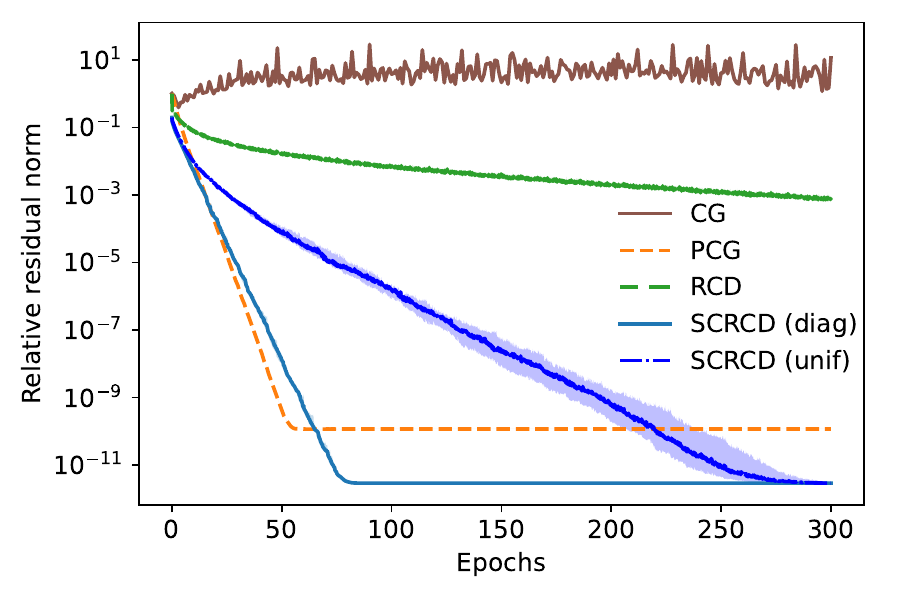}
    \includegraphics[width=0.495\linewidth, trim={0.2cm 0.3cm 0.3cm 0}, clip]{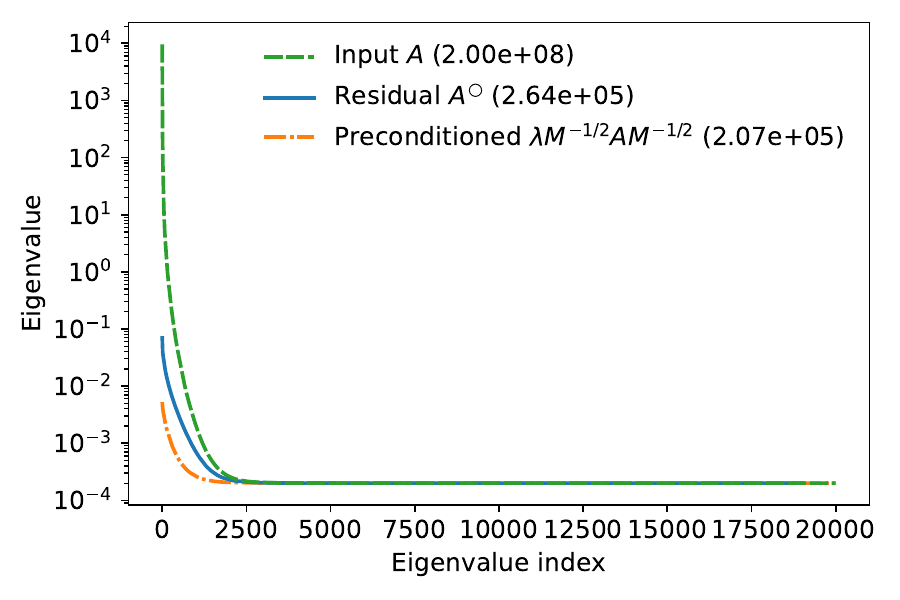}
    \vspace{-1\floatsep}
    \caption{
    For kernel matrices exhibiting rapid spectral decay, the SC-RCD method is particularly effective:
    \textbf{(Left)} Convergence trajectories and \textbf{(right)} eigenvalues for the (i) \texttt{ACSIncome} $(p=11)$, (ii) \texttt{Airlines\_DepDelay\_1M} $(p=9)$, (iii) \texttt{cod-rna} $(p=8)$, and (iv) \texttt{diamonds} $(p=9)$ datasets.
    }
    \label{fig:app_numexp_rapid}
\end{figure}

\begin{figure}[!htb]
    \centering
    \includegraphics[width=0.495\linewidth, trim={0.2cm 1cm 0.3cm 0}, clip]{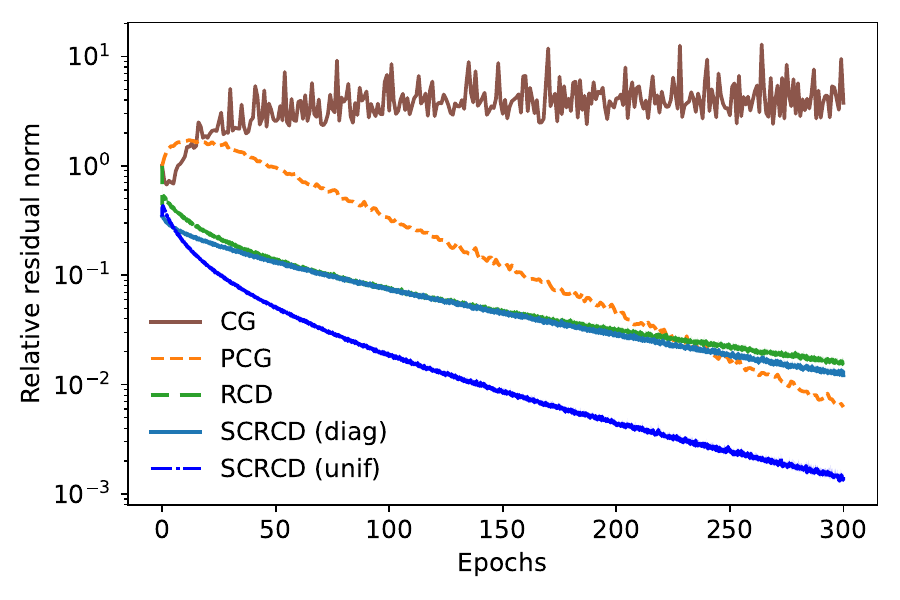}
    \includegraphics[width=0.495\linewidth, trim={0.2cm 1cm 0.3cm 0}, clip]{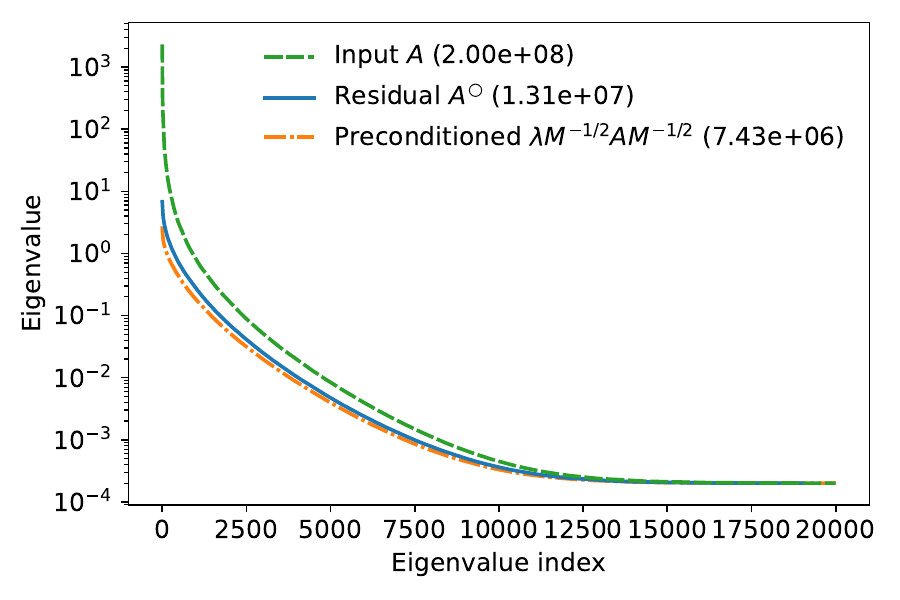}
    \includegraphics[width=0.495\linewidth, trim={0.2cm 1cm 0.3cm 0}, clip]{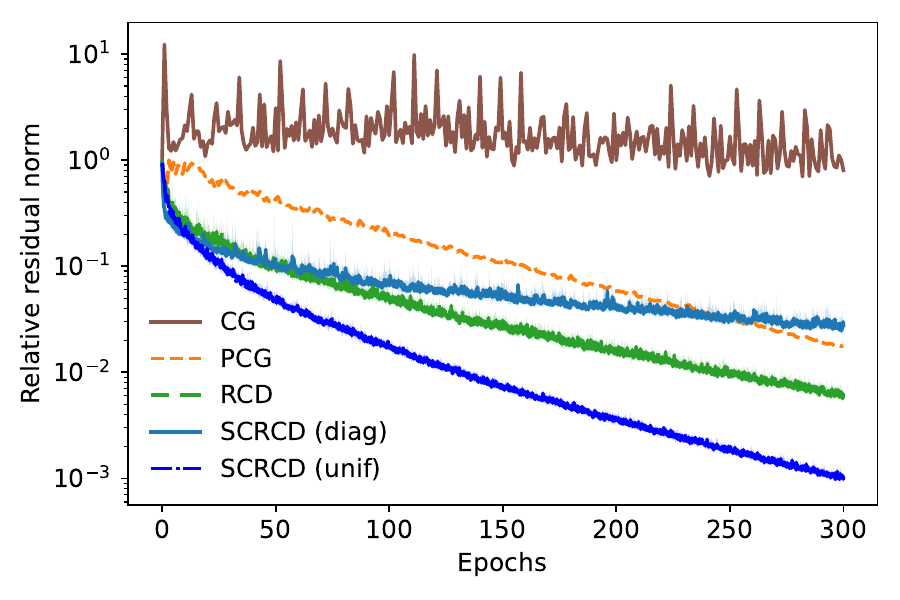}
    \includegraphics[width=0.495\linewidth, trim={0.2cm 1cm 0.3cm 0}, clip]{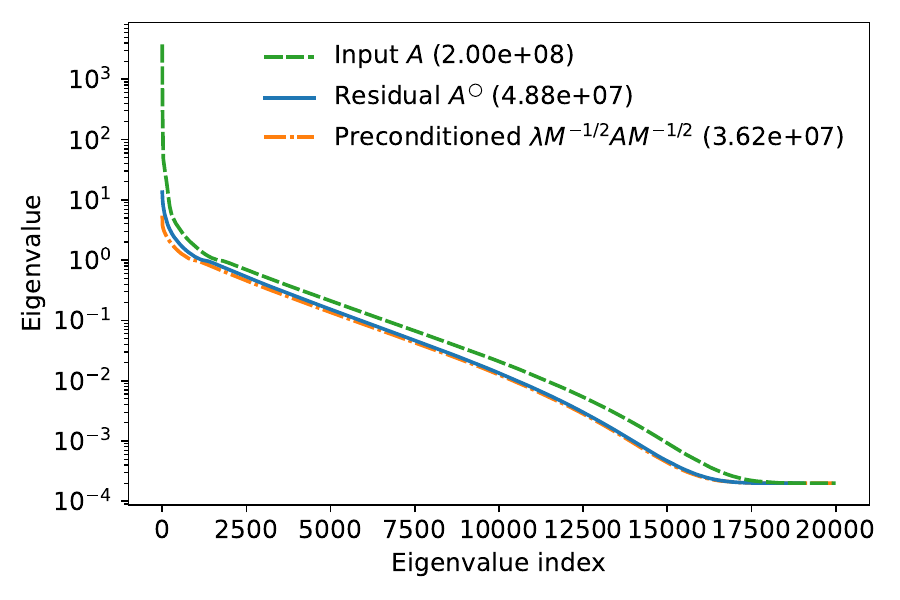}
    \includegraphics[width=0.495\linewidth, trim={0.2cm 1cm 0.3cm 0}, clip]{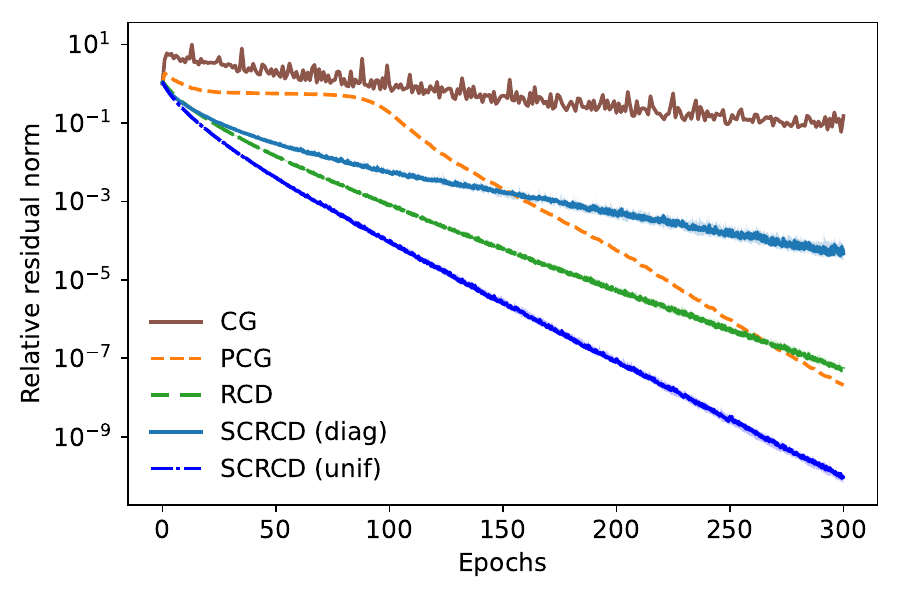}
    \includegraphics[width=0.495\linewidth, trim={0.2cm 1cm 0.3cm 0}, clip]{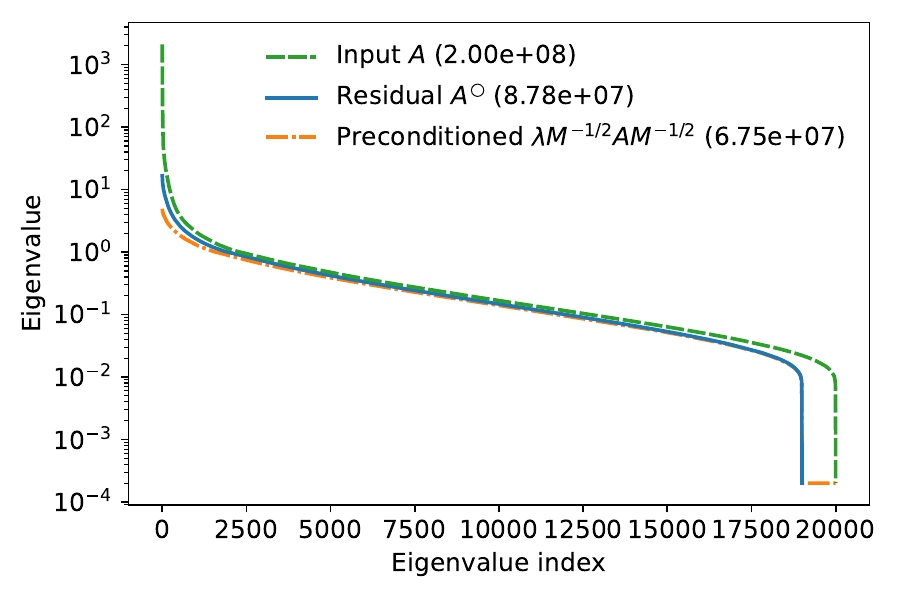}
    \includegraphics[width=0.495\linewidth, trim={0.2cm 0.3cm 0.3cm 0}, clip]{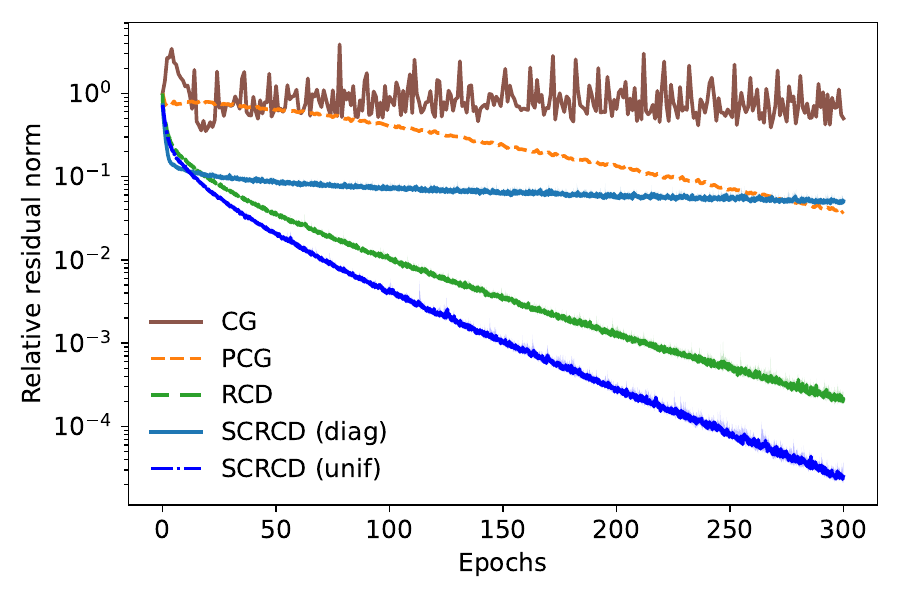}
    \includegraphics[width=0.495\linewidth, trim={0.2cm 0.3cm 0.3cm 0}, clip]{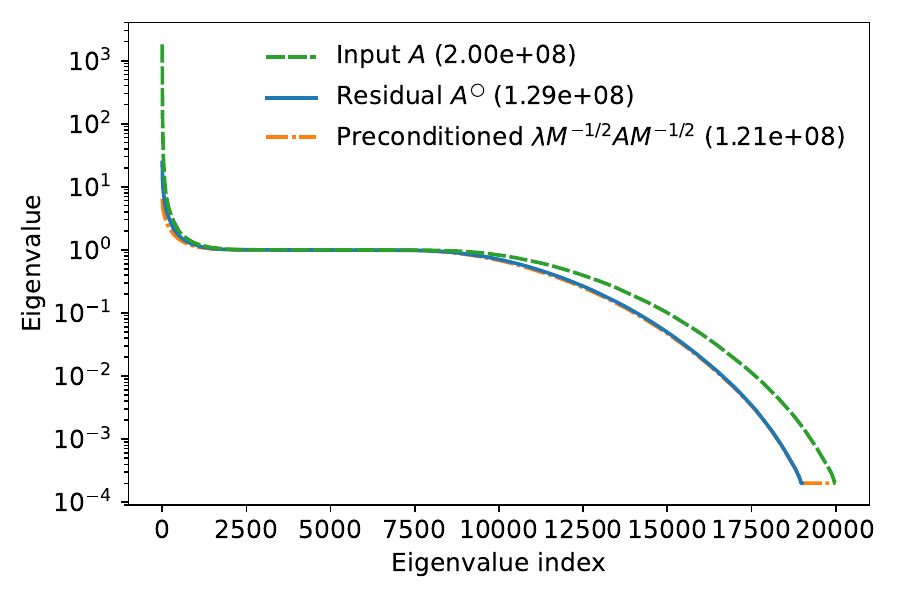}
    \vspace{-1\floatsep}
    \caption{
    For matrices with slower spectral decay, SC-RCD with uniformly sampled blocks typically works quite well:
    \textbf{(Left)} Convergence trajectories and \textbf{(right)} eigenvalues for the (i) \texttt{covtype.binary} $(p=54)$, (ii) \texttt{creditcard} $(p=29)$, (iii) \texttt{HIGGS} $(p=28)$, and (iv) \texttt{SensIT Vehicle} $(p=100)$ datasets.
    }
    \label{fig:app_numexp_slower}
\end{figure}

\end{appendix}

\end{document}